\RequirePackage{amsmath}
\documentclass[preprint,review,12pt]{elsarticle}
\usepackage[utf8]{inputenc}
\usepackage[T1]{fontenc}
\usepackage{latexsym}
\usepackage{fix-cm}
\usepackage{amssymb}

\usepackage{url}
\usepackage{hyperref}
\usepackage{graphicx}
\usepackage{tabularx}
\usepackage{array}
\usepackage{mdwmath}
\usepackage{mdwtab}
\usepackage{float}

\usepackage[cmyk]{xcolor}
\usepackage{tikz}

\usepackage[utf8]{inputenc}
\usepackage[T1]{fontenc}
\usepackage{latexsym}

\usepackage{amssymb}
\usepackage{makecell}
\usepackage{diagbox}

\usepackage{graphicx}
\usepackage{tabularx}
\usepackage{array}

\usepackage{booktabs}

\newcommand{\pc}{\mathbf{P}}

\newcommand{\pp}{\mathbb{P}}

\newtheorem{theorem}{Theorem}
\newtheorem{definition}{Definition}
\newtheorem{proposition}{Proposition}
\newproof{proof}{Proof}
\newtheorem{example}{Example}

\journal{Arxiv}

\begin{document}

\begin{frontmatter}

\title{Comparing Dependencies in Probability Theory and General Rough Sets: Part-A}

\author{\textsf{A. Mani}}
\address{Department of Pure Mathematics\\
University of Calcutta\\
9/1B, Jatin Bagchi Road\\
Kolkata(Calcutta)-700029, India\\
\texttt{Email: $a.mani.cms@gmail.com$}\\
Homepage: \url{http://www.logicamani.in}}

%\maketitle
\begin{abstract}
The problem of comparing concepts of dependence in general rough sets with those in probability theory had been initiated by the present author in some of her recent papers. This problem relates to the identification of the limitations of translating between the methodologies and possibilities in the identification of concepts. Comparison of ideas of dependence in the approaches had been attempted from a set-valuation based minimalist perspective by the present author. The deviant probability framework has been the result of such an approach. Other Bayesian reasoning perspectives (involving numeric valuations) and frequentist approaches are also known. In this research, duality results are adapted to demonstrate the possibility of improved comparisons across implications between ontologically distinct concepts in a common logic-based framework by the present author. Both positive and negative results are proved that delimit possible comparisons in a clearer way by her. 
\end{abstract}


\end{frontmatter}

%\linenumbers

\section{Introduction}\label{dpr:Sec1}

Connections between rough sets and probability theories are of much interest from theoretical and practical perspectives. A number of hybrid and analogical models for handling three-way decision making \cite{yy2012} and neo-Bayesian reasoning are known. Membership functions have also been extensively studied in rough sets (\textsf{RST}) relative to these motivations. More information can be found in  these papers \cite{ps,rk2002} for example. In simple terms, they express the degree to which an object belongs to a set. These have been interpreted in probabilistic perspectives from both Bayesian and non-Bayesian perspectives \cite{zp1984,sp94,zp5,yy2003,ps,sdt05,sdz05}. In the recent paper \cite{am9411}, these interpretations have been critically reviewed and different new problems and methodologies have been proposed by the present author. These are in the light of her work on the contamination problem \cite{am240,am3930}, advances in the philosophy of probability theory \cite{mg05,gd00b,ah2000} and possibility theory. The concept of rough membership is also generalized to granular operator spaces (see \cite{am9114,am6999,am9006,am501}) and characterized by the present author \cite{am9411}. Connections with the rough membership function based semantics \cite{mkc2014} have also been considered in the latter research. At the practical level, contamination is about making less assumptions about data and modeling vagueness as closely as is possible to the object level.

The dependence predicate \cite{am3930,am2015alc,am9501,am6000}, used instead of a probability function, is not directly comparable with the rough dependence functions \cite{am3930} as an additional layer for comparison becomes necessary. This predicate is replaced by a new dependence function based deviant probability theory for easy comparison by the present author in the research papers \cite{am9411,am1410}. This approach is shown to lead to improved methods in three-way decision making. It should be mentioned that while the concept of deviant probability has some relation to probability functions, the focus of the present paper is on measurable spaces.  
This can be helpful from the perspectives of likelihood and possibility theory \cite{dp2015,dp2016}. \emph{In the present research, the framework for the approach is improved further by using a common language and model through duality results. The interpretation and meaning are also considered in detail. It is shown that comparison of implication like operations are justified. This is because fragments of full mathematical dualities are alone meaningful in the context because of ontology. The limitations of the framework invented in this research is explored in much detail and new problems are posed.}

Both the relation-based and cover-based approach to general rough sets suffer from the problem of permitting unreal objects into the discourse - this may or may not be an issue. This is related to absence of related restrictions in information tables. For example, in big data, incompatible attributes that do not correspond to real objects can lead to wastage of resources in computing.   
While granular operator spaces and generalizations thereof can handle this and other key problems through granulations, it is a fact that not all approximations are granular. A serious method of improving the basic structures of non granular approach to rough sets  (especially covering approximation spaces) is also proposed in this research by the present author. 

In the following subsection, some of the essential background is mentioned. Granular operator spaces and concepts of rough objects are mentioned in the following section. In the third section, previous research on dependence by the present author is recapitulated. The main representation and duality results used in this paper are adapted for rough sets in the following section. Foundational aspects of covering approximation spaces and possible topologies are discussed in the fifth section. In the sixth section, the basic steps of the proposed method are outlined from a rough set perspective.  In the following section, a specific granular operator space is investigated from the duality perspective. It is shown that its natural double Heyting algebra semantics is only partly compatible with the approach. In the following section, the comparison question is taken up. Further directions are mentioned in the ninth section.     

\subsection{Some Background}

The concept of \emph{information} can also be defined in many different and non-equivalent ways. In the present author's view \emph{anything that alters or has the potential to alter a given context in a significant positive way is information}. In the contexts of general rough sets, the concept of information must have the following properties:
\begin{itemize}
\item {information must have the potential to alter supervenience relations in the contexts,}
\item {information must be formalizable and }
\item {information must generate concepts of roughly similar collections of properties or objects.}
\end{itemize}

The concept of \emph{information} in the contexts of general probability (subjective and measure-theoretic), must have the following properties:
\begin{itemize}
\item {information must have the potential to alter uncertainty relations in the context,}
\item {information must be formalizable, }
\item {information must have temporal content,}
\item {information must be bounded, }
\item {information must be granular, and}
\item {information must be relativizable.}
\end{itemize}

Further assumptions are common in all approaches and the above is about a minimalism. This has been indicated to suggest that comparisons may work well when ontologies are justified.   

The concept of an information system or table is not essential for obtaining a granular operator space or higher order variants thereof. But it often happens that they arise from such tables.  

Information systems or more correctly, information storage and retrieval systems (also referred to as information tables, descriptive systems, knowledge representation system) are basically representations of structured data tables. In the paper \cite{cd2017}, a critical reflection on the terminology used in rough sets and allied fields can be found with a suggestion to avoid plural meanings for the same term.  When columns for decision are also included, then they are referred to as \emph{decision tables}. Often rough sets arise from \emph{information tables} and decision tables. In the literature on artificial intelligence, database theory and machine learning, the term \emph{information system} refers to an integrated heterogeneous system that has components for collecting, storing and processing data. From a mathematical point of view, this can be described using heterogeneous partial algebraic systems. In rough set contexts, this generality has not been exploited as of this writing. 

An \emph{information table} $\mathcal{I}$, is a relational system of the form \[\mathcal{I}\,=\, \left\langle \mathfrak{O},\, \mathbb{A},\, \{V_{a} :\, a\in \mathbb{A}\},\, \{f_{a} :\, a\in \mathbb{A}\}  \right\rangle \]
with $\mathfrak{O}$, $\mathbb{A}$ and $V_{a}$ being respectively sets of \emph{Objects}, \emph{Attributes} and \emph{Values} respectively.
$f_a \,:\, \mathfrak{O} \longmapsto \wp (V_{a})$ being the valuation map associated with attribute $a\in \mathbb{A}$. Values may also be denoted by the binary function $\nu : \mathbb{A} \times \mathfrak{O} \longmapsto \wp{(V)} $ defined by for any $a\in \mathbb{A}$ and $x\in \mathfrak{O}$, $\nu(a, x) = f_a (x)$.

An information table is \emph{deterministic} (or complete) if
\[(\forall a\in At)(\forall x\in \mathfrak{O}) f_a (x) \text{ is a singleton}.\] It is said to be \emph{indeterministic} (or incomplete) if it is not deterministic that is
\[(\exists a\in At)(\exists x\in \mathfrak{O}) f_a (x) \text{ is not a singleton}.\]

Relations may be derived from information tables by way of conditions of the following form: For $x,\, w\,\in\, \mathfrak{O} $ and $B\,\subseteq\, \mathbb{A} $, $(x,\,w)\,\in\, \sigma $ if and only if $(\mathbf{Q} a, b\in B)\, \Phi(\nu(a,\,x),\, \nu (b,\, w),) $ for some quantifier $\mathbf{Q}$ and formula $\Phi$. The relational system $S = \left\langle \underline{S}, \sigma \right\rangle$ (with $\underline{S} = \mathbb{A}$) is said to be a \emph{general approximation space}. 

In particular if $\sigma$ is defined by the condition Eq.\ref{pawl}, then $\sigma$ is an equivalence relation and $S$ is referred to as an \emph{approximation space}.
\begin{equation*}\label{pawl}
(x, w)\in \sigma \text{ if and only if } (\forall a\in B)\, \nu(a,\,x)\,=\, \nu (a,\, w) 
\end{equation*}

In classical \textsf{RST}, on the power set $\wp (S)$, lower and upper
approximations of a subset $A\in \wp (S)$ operators, apart from the usual Boolean operations, are defined as per: 
\[A^l = \bigcup_{[x]\subseteq A} [x] \; ; \; A^{u} = \bigcup_{[x]\cap A\neq \varnothing } [x],\,\]
with $[x]$ being the equivalence class generated by $x\in S$. If $A, B\in \wp (S)$, then $A$ is said to be \emph{roughly included} in $B$ $(A\sqsubseteq B)$ if and only if $A^l \subseteq B^l$ and $A^u\subseteq B^u$. $A$ is roughly equal to $B$ ($A\approx B$) if and only if $A\sqsubseteq B$ and $B\sqsubseteq A$. The positive, negative and boundary region determined by a subset $A$ are respectively $A^l$, $A^{uc}$ and $A^{u}\setminus A^l$ respectively.

Given a fixed $A\in \wp(S)$, a \emph{Rough membership function} is a map $f_A: S \longmapsto [0,1]$ that are defined via \[(\forall x)\, f_A(x) = \dfrac{card([x]\cap A)}{card([x])}.\] Rough membership functions can be generalized to other general rough structures but lose many of the better properties valid in the classical context.

A \emph{cover} $\mathcal{C}$ on a set $\underline{S}$ is any sub-collection of $\wp(\underline{S})$. It is said to be \emph{proper} just in case $\bigcup{\mathcal{C}} = \underline{S}$. The tuple $\mathfrak{C} =\left\langle\underline{S},\, \mathcal{C}   \right\rangle$ is said to be a \emph{covering approximation space}.

A \emph{neighborhood operator} $n$ on a set $\underline{S}$ is any map of the form $n:\, \underline{S}\longmapsto \wp{\underline{S}}$. It is said to be 
\emph{Reflexive} if \begin{equation}(\forall x\in \underline{S}) \, x\in n(x) \tag{Nbd:Refl}\end{equation}
The collection of all neighborhoods $\mathcal{N}= \{n(x) \, : x\in \underline{S}\}$ of $\underline{S}$ will form a cover if and only if $(\forall x)(\exists y) x\in n(y)$ (anti-seriality). So in particular a reflexive relation on $\underline{S}$ is sufficient to generate a proper cover on it. Of course, the converse association does not necessarily happen in a unique way.  \index{RST!Cover-Based} More details about the   

If $\mathcal{S}$ is a cover of the set $\underline{S}$, then the \emph{neighborhood} of $x\in \underline{S}$ is defined via, \begin{equation}nbd(x)\,=\,\bigcap\{K:\,x\in K\,\in\,\mathcal{S}\} \tag{Cover:Nbd}\end{equation} 

The \emph{minimal description} of an element $x\in \underline{S}$ is defined to be the collection
\begin{equation}\mathrm{md} (x)\,=\,\{A\,:x\in A\in \mathcal{S},\, \forall{B}(x\in B\subseteq A\rightarrow
A= B)\} \tag{Cover:md}\end{equation} 

The \emph{maximal description} of an element $x\in \underline{S}$ is defined to be the collection:
\begin{equation}\mathrm{MD} (x)\,=\,\{A\,:x\in A\in \mathcal{S},\, (\forall{B \in \mathcal{S}})(x\in B\rightarrow
\sim(A\subset B))\} \tag{Cover:MD}\end{equation}

The \emph{indiscernibility} (or friends) of an element $x\,\in \underline{S}$
is defined to be \begin{equation}Fr(x)\,=\,\bigcup\{K:\,x\in K\in\mathcal{S}\} \tag{Cover:FR}\end{equation}

An element $K\in \mathcal{S}$ is said to be \emph{reducible} if and only if \begin{equation}(\forall x\in K)
K\neq\,MD(x) \tag{Cover:Red}\end{equation} The collection $\{nbd(x):\,x\in\,S\}$ will be denoted by $\mathcal{N}$.
The cover obtained by the removal of all reducible elements is called a \emph{covering reduct}.

Boolean algebra with approximation operators constitutes a semantics for classical RST (though not satisfactory). This continues to be true even when $R$ in the approximation space is replaced by any binary relation. More generally it is possible to replace $\wp (S)$ by some set with a part-hood relation and some approximation operators defined on it \cite{am240}. The associated semantic domain in the sense of a collection of restrictions on possible
objects, predicates, constants, functions and low level operations on those is 
referred to as the classical semantic domain for general RST. In contrast, the semantic domain
associated with sets of roughly equivalent or relatively indiscernible objects with
respect to this domain is a \emph{rough semantic domain}. Actually many other
semantic domains, including hybrid semantic domains, can be generated and have been used for example in choice-inclusive semantics \cite{am99}, but these two broad domains will always be - though not necessarily with a nice correspondence between the two. 

The basic problem of contamination is that of mix up of semantic notions during modeling the dynamics within a specific semantic domain. At the practical level it can be in using semantic aspects of the classical semantic domain in modeling interaction in the rough semantic domain - if somebody is reasoning about aggregating two rough objects, then they are not likely to know about the classical ontology associated with their awareness ($\sqcup$ is an example of such an operation) and operations like union and intersection. But at the theoretical level many models assume as much. In other words in classical semantic domain operations and predicates used to describe semantics of specific rough semantic domains may not exist in the rough domain in the first place. For more details the reader is referred to the research papers \cite{am501,am240,am3600,am3930}.

\paragraph{Topology}
A \emph{totally disconnected topological space} is a topological space in which all connected components are singletons. Examples include discrete spaces, $Q$, $\Re \setminus Q$ with the usual topology, $Q_p$ (set of p-adic numbers with usual topology), Baire spaces, Stone spaces (compact, totally disconnected Hausdorff spaces) and Cantor spaces. 

It is well known that any two non-empty compact Hausdorff spaces without isolated points and of zero-dimension (having countable bases consisting of clopen sets) are homeomorphic to each other.

\begin{definition}
A \emph{Priestley Space} is a structure of the form $\left\langle X, \leq , \tau_X   \right\rangle$ in which all of the following hold:
\begin{itemize}
\item {$\left\langle X, \leq   \right\rangle$ is a partially ordered set,}
\item {$\left\langle X, \tau_X   \right\rangle$ is a topological space in which for all $a, b\in X$ satisfying $a\nleq b$, there exists a clopen increasing set $K$ for which $a\in K$ and $b\notin K$ holds and  }
\item {the topology is compact.}
\end{itemize}
So it is a compact totally order-disconnected topological space. The set of clopen increasing sets is denoted by $inclop(X)$.

An \emph{Esakia space} is a Priestley space in which every order ideal generated by clopen sets is again clopen 
\end{definition}

\subsubsection{Sigma Algebras}

\begin{definition}
A \emph{concrete $\sigma$-algebra} (or a \emph{measurable space}) is a tuple of the form $\left\langle X, \mathcal{S}   \right\rangle$ with $X$ being a set and $\mathcal{S}$ is a collection of  subsets of $X$  that is closed under countable union, countable intersection and complementation.
\end{definition}

\begin{definition}
An abstract $\sigma$-algebra $\mathcal{S} = \left\langle \underline{\mathcal{S}}, \cup, \cap, ^c , \emptyset    \right\rangle$ is an infinitary algebra with $\underline{\mathcal{S}}$ being a collection of  sets  and the following hold:
\begin{itemize}
\item {$\mathcal{S}$ is a bounded, complemented distributive lattice with least element $\emptyset$.}
\item {$(\forall x_i) \, \cup_{i=1}^{n} x_i\,\&\, \cap_{i=1}^{n} x_i $ are defined.}
\end{itemize}
\end{definition}

In general, complete Boolean algebras need not be $\sigma$-algebras because infinite distributivity may not hold in the former. Let $\mathcal{S}$ be the $\sigma$-algebra of Lebesgue measurable sets on the unit interval $[0,\, 1]$. Form the quotient $\mathcal{S}_o$ by identifying sets that differ by a set of Lebesgue measure $0$ in $\mathcal{S}$. Then $\mathcal{S}_o$ is a complete Boolean algebra that is not a $\sigma$-algebra. 

The collection $\mathfrak{S}(S)$ of all $\sigma$-algebras on a set $S$ is lattice ordered by inclusion, is bounded  and includes the following sigma-algebras:  
\begin{itemize}
\item {The least or trivial $\sigma$-algebra over $S$ is $\{\emptyset, S\}$.}
\item {The power set $\wp (S)$ of $S$, is the discrete $\sigma$-algebra.}
\item {The collection $\{\emptyset, A, A^c, S\}$ is the simple $\sigma$-algebra generated by the subset $A$.}
\item {The collection of subsets of $S$ that are countable (finite or infinite) or whose complements are countable is a $\sigma$-algebra. This is the $\sigma$-algebra generated by the singletons of $S$. It coincides with the powerset $\wp(S)$ if $S$ is not uncountable.}
\item {The collection of all unions of sets in a countable partition of $S$ is a $\sigma$-algebra.}
\end{itemize}

\begin{proposition}
The covers $\mathcal{K}$ used in a forming a covering approximation space $ S, \mathcal{K}$ does not include the empty set. So no cover can contain a $\sigma$-algebra, but every cover generates a unique $\sigma$-algebra in which it is included.     
\end{proposition}

Let $\mathcal{S}$ be the $\sigma$-algebra of Lebesgue measurable sets on the unit interval $[0,\, 1]$. Form the quotient $\mathcal{S}_o$ by identifying sets that differ by a set of Lebesgue measure $0$ in $\mathcal{S}$. Then $\mathcal{S}_o$ is a complete Boolean algebra that is not a $\sigma$-algebra. 

Let $\mathcal{S}$ be an abstract $\sigma$-algebra, a $\sigma$-ideal of $\mathcal{S}$ is a subset $\mathcal{K}\subseteq \mathcal{S}$, that is closed under countable unions and is an order-ideal relative to set-inclusion. For a $\sigma$-ideal $\mathcal{K}$, let $A\sim B$ for any $A, B\in \mathcal{S}$ if and only if $A\Delta B \in \mathcal{K}$. $\sim$ is an equivalence on $\mathcal{S}$. On the quotient $\mathcal{S\mid K}$ an abstract $\sigma$-algebra structure can be directly induced.

The Loomis-Sikorski theorem generalizes the Stone representation theorem in the following way:

\begin{theorem}
Let $\mathcal{S}$ be an abstract $\sigma$-algebra, then there exist concrete $\sigma$-algebras of the form $S, \mathcal{B}$
and a $\sigma$-ideal $\mathcal{K}$ of $\mathcal{B}$ such that $\mathcal{S}$ is isomorphic to $\mathcal{B\mid K}$. 
\end{theorem}

The theorem can be upgraded to a duality by adding the missing morphisms and functors.

\begin{theorem}
Let $\mathcal{S}$ be a $\sigma$-complete Boolean algebra (Boolean algebras in which every countable collections of subsets have an upper bound), then there exist concrete $\sigma$-algebras of the form $S, \mathcal{B}$
and a $\sigma$-ideal $\mathcal{K}$ of $\mathcal{B}$ such that $\mathcal{S}$ is isomorphic to $\mathcal{B\mid K}$. 
\end{theorem}

From this result it is also possible to show by a contradiction argument that 
\begin{proposition}
$\sigma$-complete Boolean algebras without atoms that satisfy the countable chain condition have no $\sigma$-complete ultrafilters and cannot be represented as an algebra of sets.
\end{proposition}

The $\sigma$-algebra of a measure space has a natural pseudo-metric $\varrho$ associated - the measure of the symmetric difference between two subsets being the pseudo metric distance between the subsets in question defined as   

\[(\forall A, B\in \mathcal{S})\,\varrho(A, B) = \mu (A\Delta B)\]

\section{Granular Operator Spaces and Variants}

Granular operator spaces and related variants are not necessarily basic systems in the context of application of general rough sets. They are powerful abstractions for handling semantic questions, formulation of semantics and the inverse problem 

\begin{definition}\label{gos}
A \emph{Granular Operator Space}\cite{am6999} $S$ is a structure of the form $S \, =\, \left\langle \underline{S}, \mathcal{G}, l , u\right\rangle$ with $\underline{S}$ being a set, $\mathcal{G}$ an \emph{admissible granulation}(defined below) over $S$ and $l, u$ being operators $:\wp(\underline{S})\longmapsto \wp(\underline{S})$ ($\wp(\underline{S})$ denotes the power set of $\underline{S}$) satisfying the following ($\underline{S}$ will be replaced with $S$ if clear from the context. \textsf{Lower and upper case alphabets will both be used for subsets} ):

\begin{align*}
a^l \subseteq a\,\&\,a^{ll}\, =\,a^l \,\&\, a^{u} \subset a^{uu}  \\
(a\subseteq b \longrightarrow a^l \subseteq b^l \,\&\,a^u \subseteq b^u)\\
\emptyset^l\, =\,\emptyset \,\&\,\emptyset^u\, =\,\emptyset \,\&\,\underline{S}^{l}\subseteq S \,\&\, \underline{S}^{u}\subseteq S.
\end{align*}

Here, \emph{Admissible granulations} are granulations $\mathcal{G}$ that satisfy the following three conditions ($t$ is a term operation formed from the set operations $\cup, \cap, ^c, 1, \emptyset$):

\begin{align*}
(\forall a \exists
b_{1},\ldots b_{r}\in \mathcal{G})\, t(b_{1},\,b_{2}, \ldots \,b_{r})=a^{l} \\
\tag{Weak RA, WRA} \mathrm{and}\: (\forall a)\,(\exists
b_{1},\,\ldots\,b_{r}\in \mathcal{G})\,t(b_{1},\,b_{2}, \ldots \,b_{r}) =
a^{u},\\
\tag{Lower Stability, LS}{(\forall b \in
\mathcal{G})(\forall {a\in \wp(\underline{S}) })\, ( b\subseteq a\,\longrightarrow\, b \subseteq a^{l}),}\\
\tag{Full Underlap, FU}{(\forall
a,\,b\in\mathcal{G})(\exists
z\in \wp(\underline{S}) )\, a\subset z,\,b \subset z\,\&\,z^{l}\, =\,z^{u}\, =\,z,}
\end{align*}
\end{definition}

\begin{flushleft}
\textbf{Remarks}: 
\end{flushleft}
\begin{itemize}
\item {The concept of admissible granulation was defined for \textsf{RYS} in \cite{am240} using parthoods instead of set inclusion and relative to \textsf{RYS}, $\pc\, =\,\subseteq$, $\pp\, =\,\subset$. }
\item {The conditions defining admissible granulations mean that every approximation is somehow representable by granules in a set theoretic way, that granules are lower definite, and that all pairs of distinct granules are contained in definite objects.}
\end{itemize}

On $\wp(\underline{S})$, the relation $\sqsubset$ is defined by \begin{equation}A \sqsubset B \text{ if and only if } A^l \subseteq B^l \,\&\, A^u \subseteq B^u.\end{equation} The rough equality relation on $\wp(\underline{S})$ is defined via $A\approx B \text{ if and only if } A\sqsubset B  \, \&\,B \sqsubset A$. 

Regarding the quotient $\wp(\underline{S})|\approx$ as a subset of $\wp(\underline{S})$, the order $\Subset$ will be defined as per \begin{equation}\alpha \Subset \beta \text{ if and only if } \alpha^l \subseteq \beta^l \,\&\, \alpha^u \subseteq \beta^u.\end{equation} Here $\alpha^l$ is being interpreted as the lower approximation of $\alpha$ and so on. $\Subset$ will be referred to as the \emph{basic rough order}.

\begin{definition}
By a \emph{roughly consistent object}\index{Object!Roughly Consistent} will be meant a set of subsets of $\underline{S}$ of the form  $H\, =\,\{A ; (\forall B\in H)\,A^l =B^l, A^u\, =\,B^u \}$. The set of all roughly consistent objects is partially ordered by the inclusion relation. Relative this maximal roughly consistent objects will be referred to as \emph{rough objects}\index{Object!Rough}. By \emph{definite rough objects}\index{Object!Definite}, will be meant rough objects of the form $H$ that satisfy 
\begin{equation}(\forall A \in H) \, A^{ll}\, =\,A^l \,\&\, A^{uu}\, =\,A^{u}. \end{equation} 
\end{definition}

\begin{proposition}
$\Subset$ is a bounded partial order on $\underline{S}|\approx$. 
\end{proposition}
\begin{proof}
Reflexivity is obvious.  If $\alpha \Subset \beta$ and $\beta \Subset \alpha$, then it follows that $\alpha^l\, =\,\beta^l$ and $\alpha^u\, =\,\beta^u$ and so antisymmetry holds. 

If $\alpha \Subset \beta$, $\beta \Subset \gamma$, then the transitivity of set inclusion induces transitivity of $\Subset$.
The poset is bounded by $0\, =\,(\emptyset , \emptyset)$ and $1\, =\,(S^l , S^u)$. Note that $1$ need not coincide with $(S, S)$. 
\qed
\end{proof}

On $\wp(\underline{S})$, the relation $\sqsubset$ is defined by \[A \sqsubset B \text{ if and only if } A^l \subseteq B^l \,\&\, A^u \subseteq B^u.\] The rough equality relation on $\wp(\underline{S})$ is defined via $A\approx B \text{ if and only if } A\sqsubset B  \, \&\,B \sqsubset A$. 

Rough membership functions make essential use of ideas of association of points in $S$ with granules - this is true in the classical case, but as the underlying structures are generalized the idea can be abandoned 

\begin{definition}
Given an admissible granulation $\mathcal{G}$ on $S$, \emph{granular neighborhood maps} will be maps of the form $\gamma_{t} : S \longmapsto \wp(S)$ ($t$ being a type) definable as per the following schemas: 
\begin{align*}
\tag{Cap} \gamma_{\cap}(x) = \cap \{g : x\in g \in \mathcal{G}\}   \\
\tag{Cup} \gamma_{\cup}(x) = \cup \{g : x\in g \in \mathcal{G}\} \\
\tag{Choice} \gamma_{ch}(x) = \lambda \{g : x\in g \in \mathcal{G}\}\\
\tag{QIu} \gamma_{ch}(x) = (\cap \{g : x\in g \in \mathcal{G}\})^u\\
\tag{QIu} \gamma_{ch}(x) = (\cup \{g : x\in g \in \mathcal{G}\})^l\\
\tag{pCap} \gamma_{p}(x) = 
\begin{cases}
\cap \{g : x\in g \in \mathcal{G}\}, & \text{if RHS is in } \mathcal{G}  \\
\text{undefined }, & \mathrm{else}.
\end{cases}
\end{align*}
with $\lambda $ being a choice map  $:\wp (\mathcal{G}) \longmapsto \mathcal{G}$ satisfying $(\forall H \in \wp (\mathcal{G}))\, \lambda(H) \in H$ will be termed point maps. 
\end{definition}

\begin{definition}
A General Rough Membership Function $\omega_{\gamma}$ is a function $: {S} \times \wp(S)\,\longmapsto [0,1]$ defined as follows:   
\begin{equation}
(\forall x \in S)(\forall A\in \wp (S))\, \omega_{\gamma} (x, A) = \dfrac{\textsf{Card}(\gamma(x) \cap A)}{\textsf{Card}(\gamma(x))}.                                                                                                                                                                                                                                                                                 \end{equation}
\end{definition}

\begin{definition}\label{baserel}
The general rough membership induces three relations $R_{A\omega}$ on $S$ for each $A\in \wp(S)$, $\sim_{x\omega}$ on $\wp(S)$ for each $x\in \underline{S}$ and $\backsimeq$ on $S\times \wp(S)$ defined as below:
\begin{align}
\text{For } x, y\in S\, \& \, A\in \wp(S)\; R_{A\omega}xy \leftrightarrow \omega_{\gamma}(x, A) =  \omega_{\gamma}(y, A) \\
\text{For } x, \in S\, \& \, A, B\in \wp(S)\; \sim_{x\omega}A B \leftrightarrow \omega_{\gamma}(x, A) =  \omega_{\gamma}(x, B) \\
\text{For } x, y \in S\, \& \, A, B\in \wp(S)\; \backsimeq_{\omega}(x,A)(y, B) \leftrightarrow \omega_{\gamma}(x, A) =  \omega_{\gamma}(y, B)
\end{align}
\end{definition}

\begin{proposition}
 In the context of Def.\ref{baserel}, all of the three relations are equivalences.
\end{proposition}

The order induced on the quotients (regarded as set of some subsets of the power set of their domains) by the inclusion order on power sets of their domain are of natural interest. They can also help in solving the following long problem:

\begin{flushleft}
{\bf Problem:}\end{flushleft}  \textsf{Which choices of subsets of $\wp(S)$ and elements of $S$ have better likelihood of coverage of the actual contamination-free rough semantics than others?}

In the context of this class of problems, the concept of dependence space as in \cite{nov} is related as a special case from a mathematical perspective.

\begin{definition}
A \emph{dependence space} is a tuple of the form $\left\langle A, F \right\rangle$ with $A$ being a set and $F$ a congruence on the semilattice $\left\langle \wp (A), \cup \right\rangle$.  
\end{definition}

\begin{theorem}
A general rough membership function has the following properties:
\begin{align*}
\tag{Monotony} (\forall x \in S)(\forall A, B\in \wp(S))(A\subseteq B \longrightarrow \omega_{\gamma} (x, A) \leq \omega_{\gamma} (x, B))\\
\tag{Empty Set} (\forall x \in S)\,\omega_{\gamma} (x, \varnothing) =0\\
\tag{Top}  (\forall x \in S)\,\omega_{\gamma} (x, S) =1\\
\tag{Granular Equality} (\forall x, y\in S)(\forall A\in \wp(S))(\gamma(x) = \gamma(y) \longrightarrow \omega_\gamma (x, A) = \omega_\gamma (y, A))\\
\tag{G-Monotony} (\forall x, y \in S)(\forall A \in \wp(S))(\gamma(x) \subseteq \gamma (y) \longrightarrow \omega_{\gamma} (x, A) \leq \omega_{\gamma} (y, A))
\end{align*}
\end{theorem}

\begin{proof}
The proof consists in direct verification and is not hard.
\qed
\end{proof}

Most of the better properties that hold for rough membership functions in the classical context fail to hold in general contexts.

The concept of \emph{general granular operator spaces} had been introduced by the present author \cite{am9114,am6900} as a proper generalization of that of granular operator spaces. The main difference is in the replacement of $\subset$ by arbitrary \emph{part of} ($\pc$) relations in the axioms of admissible granules and inclusion of $\pc$ in the signature of the structure.
\begin{definition}
A \emph{General Granular Operator Space} (\textsf{GSP}) $S$ shall be a structure of the form $S \, =\, \left\langle \underline{S}, \mathcal{G}, l , u, \pc \right\rangle$ with $\underline{S}$ being a set, $\mathcal{G}$ an \emph{admissible granulation}(defined below) over $S$, $l, u$ being operators $:\wp(\underline{S})\longmapsto \wp(\underline{S})$ and $\pc$ being a definable binary generalized transitive predicate (for parthood) on $\wp(\underline{S})$ satisfying the same conditions as in Def.\ref{gos} except for those on admissible granulations (generalized transitivity can be any proper nontrivial generalization of parthood (see \cite{am501}). $\pp$ is  proper parthood (defined via $\pp ab$ iff $\pc ab \,\&\,\neg \pc ba$) and $t$ is a term operation formed from set operations):

\begin{align*}
(\forall x \exists
y_{1},\ldots y_{r}\in \mathcal{G})\, t(y_{1},\,y_{2}, \ldots \,y_{r})=x^{l} \\
\tag{Weak RA, WRA} \mathrm{and}\: (\forall x)\,(\exists
y_{1},\,\ldots\,y_{r}\in \mathcal{G})\,t(y_{1},\,y_{2}, \ldots \,y_{r}) =
x^{u},\\
\tag{Lower Stability, LS}{(\forall y \in
\mathcal{G})(\forall {x\in \wp(\underline{S}) })\, ( \pc yx\,\longrightarrow\, \pc yx^{l}),}\\
\tag{Full Underlap, FU}{(\forall
x,\,y\in\mathcal{G})(\exists
z\in \wp(\underline{S}) )\, \pp xz,\,\&\,\pp yz\,\&\,z^{l}\, =\,z^{u}\, =\,z,}
\end{align*}
\end{definition}

It is sometimes more convenient to use only sets and subsets in the formalism as these are the kinds of objects that may be observed by agents and such a formalism would be more suited for reformulation in formal languages. This justifies the introduction of higher order granular operator spaces \cite{am9006} by the present author.

\begin{definition}
An element $x\in\mathbb{S}$ will be said to be \emph{lower definite} (resp. \emph{upper definite}) if and only if $x^l\, =\,x$ (resp. $x^u\, =\,x$) and \emph{definite}, when it is both lower and upper definite. $x\in \mathbb{S}$ will also be said to be \emph{weakly upper definite} (resp \emph{weakly definite}) if and only if $ x^u\, =\,x^{uu} $ (resp $ x^u\, =\,x^{uu} \,\&\, x^l =x$ ). Any one of these five concepts may be chosen as a concept of \emph{crispness}. 
\end{definition}

\subsection{Rough Objects}

The concept of rough objects must necessarily relate to some of the following:
\begin{itemize}
\item {object level properties of approximations in a suitable semantic domain,}
\item {object level properties of discernibility in a suitable semantic domain,}
\item {object level properties of indiscernibility in a suitable semantic domain,}
\item {properties of abstractions from approximations,}
\item {properties of abstractions of indiscernibility, or}
\item {some higher level semantic features (possibly constructed on the basis of some assumptions about approximations).}
\end{itemize}

A rough object cannot be known exactly in the rough semantic domain, but can be represented through various means. This single statement hides deep philosophical aspects that are very relevant in practice if realizable in concrete terms.  The following concepts of \emph{rough objects} have been either considered in the literature (see \cite{am240,am9006}) or are reasonable concepts:
\begin{itemize}
\item[RL] {$x\in \mathbb{S}$ is a lower rough object if and only if $\neg (x^l\, =\,x) $. }
\item[RU] {$x\in \mathbb{S}$ is a upper rough object if and only if $\neg (x\, =\,x^u) $. }
\item[RW] {$x\in \mathbb{S}$ is a weakly upper rough object if and only if $\neg (x^u\, =\,x^{uu}) $. }
\item[RB] {$x\in \mathbb{S}$ is a rough object if and only if $\neg (x^l\, =\,x^u) $. The condition is equivalent to the boundary being nonempty. }
\item[RD] {\emph{Any pair of definite elements} of the form $(a , b)$ satisfying $a < b $}
\item[RP] {\emph{Any distinct pair of elements} of the form $(x^l ,x^u)$.}
\item[RIA] {Elements in an \emph{interval of the form} $(x^l, x^u)$.}
\item[RI] {Elements in an \emph{interval of the form} $(a, b)$ satisfying $a\leq b$ with $a, b$ being definite elements.}
\item[ET] {In esoteric rough sets \cite{am24}, triples of the form $(x^l, x^{lu}, x^u)$ can be taken as rough objects.} 
\item[RND] {A \emph{non-definite element in a RYS}(see \cite{am240}), that is an $x$ satisfying $\neg \pc x^u x^l   $. This can have a far more complex structure when multiple approximations are available.}
\item[ROP]{If a weak negation or complementation $^c$ is available, then orthopairs of the form $(x^l, x^uc)$ can also be taken as representations of \emph{rough objects}.}
\end{itemize}

All of the above concepts of a rough object except for the last two are directly usable in a higher granular operator space. 

The positive region of a $x\in \mathbb{S}$ is $x^{l}$, while its negative region is $x^{uc}$ -- this region is independent from $x$ in the sense of attributes being distinct, but not in the sense of derivability or inference by way of rules. These derived concepts provide additional approaches to specifying subtypes of rough objects and related decision making strategies.

\begin{itemize}
\item {$POS(x)\, =\,x^{l} $ and $NEG(x)= x^{uc} $ by definition.}
\item {$x$ is \emph{roughly definable} if $POS(x)\neq\emptyset $ and $NEG(x)\neq \emptyset $}
\item {$x$ is \emph{externally undefinable} if $POS(x)\neq \emptyset $ and $NEG(x)=\emptyset $}
\item {$x$ is \emph{internally undefinable} if $POS(x)=\emptyset $ and $NEG(x)\neq \emptyset $}
\item {$x$ is \emph{totally undefinable} if $POS(x)=\emptyset $ and $NEG(x)=\emptyset $}
\end{itemize}

It makes sense to extend this conception in general rough contexts where there exist subsets $x$ for which $x^{u} \subset x^{uu}$ is possible. In the context of cover based rough sets, the cover is often adjoined to $NEG ,\, POS$ as a subscript as in $NEG_{\mathcal{S}}$ and $POS_{\mathcal{S}}$  respectively. 

\begin{definition}
By the strong negative region associated with a subset $x\in\mathbb{S}$ will be meant the element 
\[SNEG (x)\, =\,x^{uuc} \]
\end{definition}

\section{Dependence}

In this section, the concepts of rough and probabilist dependence considered by the present author in her earlier research \cite{am9411,am3930,am9501} are summarized.

\subsection{General Rough Dependence}

In this subsection, concepts of \emph{rough dependence} in general rough set theory including the ones introduced earlier by the present author \cite{am9411,am3930,am9501,am6000} are explained.  The idea of \emph{rough dependence} is in relation to the functional aspects of rough semantics. At a minimal level the relation between a pair of rough or definite objects is representable by the rough objects that are in common. The concept of dependence spaces \cite{nov} relates to information systems from the point of view of reduct computation. It is very distinct from the knowledge related approach here, but is related to the particular sub problems.

In the granular operator spaces, degrees of rough dependence can be defined as below (note that it is assumed that $\pc = \subseteq$, $\pp = \subset$, $\nu(S)$ is the collection of definite objects, $\oplus = \cup$, $\odot = \cap$, $0=\varnothing, 1= S$ and $\tau(S)$ is a granulation on $S$. ) 

\begin{definition}
The $\tau \nu$-\emph{infimal degree of dependence} $\beta_{i \tau \nu}$ of a subset $A$ on $B$ is defined by 
\begin{equation}
\beta_{i \tau \nu} (A,\, B)\,=\,\inf _{\nu (S) }\,\oplus \,\{C\,:\,C\in \tau(S) \, \&\,\pc C A \,\& \, \pc C B\}. 
\end{equation}
The infimum refers to the largest $\nu(S)$ element contained in the union.

The $\tau \nu$-\emph{supremal degree of dependence} $\beta_{s \tau \nu}$ of a subset $A$ on $B$ is defined by
\begin{equation}
\beta_{s \tau \nu} (A,\, B)\,=\,\sup _{\nu (S) }\,\oplus \,\{C\,:\,C\in \tau(S) \,\&\, \pc C A \,\& \, \pc C B\}.                                                                                                     
\end{equation}
The supremum refers to the least $\nu(S)$ element containing the sets.
These concepts can extended to more general structures like \textsf{RYS} directly.
\end{definition}

\begin{definition}
Two elements $x,\, y$ in a granular operator space $S$ will be said to be \emph{PN-independent} $ I_{PN}(xy)$ if and only if 
\begin{equation} 
x^{l}\,\subseteq \,y^{uc}\; \&\;y^{l}\,\subseteq \,x^{uc} 
\end{equation}
and two elements $x,\, y$ in a granular operator space $S$ will be said to be \emph{PN-dependent} $\varsigma_{PN}(xy)$ if and only if  
\begin{equation} 
x^{l}\,\nsubseteq \,y^{uc}\; \&\;y^{l}\,\nsubseteq \,x^{uc} .  
\end{equation}
\end{definition}

\begin{theorem}
In classical rough sets with $\tau(S)\,=\, \mathcal{G}(S)$ - the granulation of $S$ (set of equivalence classes) and $\nu(S)\,=\, \delta_{l}(S)$ - the set of lower definite elements (explicit references to these are dropped in the following). The first property below allows the subscripts $i , \, s$ on $\beta$ to be omitted.
\begin{align} 
\beta_{i}x y\,=\, x^{l}\, \cap\,y^ l \,=\,\beta_{s}x y  \\
\beta x x \,=\, x^l. \; \beta x y \,=\, \beta y x  \\ 
\beta (\beta x y) x \, =\, \beta x y \\
\pc (\beta x y)(\beta x (y \oplus z)) \\
(\pc x y \longrightarrow \beta x y = x^{l}) \\
(x\odot y\,=\, 0 \,\longrightarrow \, \beta_{i} x y \,=\, 0) \\ 
\beta 0 x = 0 \,; \;\, \beta x 1 = x^{l} \\
(\pc y^{l} z \longrightarrow \pc (\beta x y)(\beta x z)) \\
\beta x y \,=\, \beta x^{l} y^{l} \,=\, \beta x y^{l} 
\end{align}
The converse of $(x\odot y\,=\, 0 \,\longrightarrow \, \beta_{i} x y \,=\, 0)$ is not true in general. 
\end{theorem}

\begin{theorem}
Semantics of classical rough sets over the classical semantic domain can be formulated using the operations $\cap,\, c ,\, \beta $ on the power-set of $S$. Reference to lower and upper approximation operators can thus be avoided.
\end{theorem}

The last result means that rough dependence can be used effectively as an alternative to approximations.

\subsection{Extension to Granular Operator Spaces}

The results for classical rough sets cannot be expected to generalize to granular operator spaces as many of the nicer order-theoretic conditions that are true in the former case do not hold in granular operator spaces. Still the resulting dependence based structure can be used to enhance the anti-chain based semantics invented by the present author \cite{am9114,am6999}.

In a granular operator space, if $\tau(S) = \mathcal{G}(S) $ and $\nu(S)= \delta (S) $ is the set of definite elements, then the definition of rough dependence specializes to the following:

\begin{definition}

The \emph{infimal degree of dependence} $\beta_{i }$ of $A$ on $B$ is defined by 
\begin{equation}
\beta_{i } (A,\, B)\,=\,\inf _{\delta (S) }\,\bigcup \,\{C\,:\,C\in \mathcal{G}(S) \, \&\, C\subseteq A \,\& \,  C\subseteq B\}. 
\end{equation}
The infimum refers to the largest definite element contained in the union.

The \emph{supremal degree of dependence} $\beta_{s }$ of $A$ on $B$ is defined by 
\begin{equation}
\beta_{s } (A,\, B)\,=\,\sup _{\delta (S) }\,\bigcup \,\{C\,:\,C\in \mathcal{G}(S) \,\&\,  C\subseteq A \,\& \, C \subseteq B\}. 
\end{equation}
The supremum refers to the least definite element containing the sets.
\end{definition}

\begin{theorem}
Under the above assumptions,
\begin{align}
\beta_i x y \,=\, \beta_i y x \\ 
\beta_i x y \subseteq \beta_i x (y \cup z) \\
( x \subseteq y \longrightarrow \beta_i x y = \beta_i x x) \\
(x\cap y\,=\, 0 \,\longrightarrow \, \beta_{i} x y \,=\, 0) \\
\beta_i 0 x = 0 \,; \;\, \beta_i x 1 = \beta_i x x .
\end{align}
The converse of $(x\cap y\,=\, 0 \,\longrightarrow \, \beta_{i} x y \,=\, 0)$ does not hold in general.
\end{theorem}

The main difference with the classical case is that no representation of the operators $l, u$ are assumed and the consequences of the granularity assumptions cannot be integrated with $\cup$ and $ \cap$. So all of the following are possible:
\begin{align}
\beta_{i}x y\,\neq\, (x\, \cap\,y)^ l \\
\beta_{s}x y \nsubseteq x\cup y \\
\beta_i x x \,\neq\, x^l \neq \beta_s x x 
\end{align}

\section{Dependence Based Probability}

In this subsection the dependence predicate based approach \cite{am3930,am9501,am6000} due to the present author which in turn is an abstraction of the approach in \cite{bd2010} is improved. Dependence of events are defined over probability spaces of the form $(X,\,\mathcal{S} ,\, p)$ with $X$ being a set, $\mathcal{S}$ being a $\sigma$-algebra over $X$ and $p$ being a probability function. It is possible to extend the definition to systems of probability measures over the same $\sigma$-algebra. 

\begin{definition}
A dependence function over the probability space is a function $\delta :\, {\mathcal{S}}^{2}\,\longmapsto \, \Re $  defined by 
\begin{equation} 
\delta (x, \, y)\,=\, p(x\cap y)\,-\, p(x)\,\cdot \, p(y)                                                                                                                                                                                                                                                         \end{equation}
\end{definition}

Two events $x,\, y \in \mathcal{S} $ are \emph{mutually exclusive} if and only if $x\cap y\,=\, \varnothing $. This concept involves temporality and it is not possible to speak of mutual exclusivity without reference to temporality. Related temporality is correctly at the meta level and this has no parallel with the situation in rough sets and must be taken into account for proper comparisons. The concept is extensible to countable collections of sets of events in a natural way. Weaker forms of mutual exclusivity are also of interest in the comparison perspective:

\begin{definition}
$a,\, b\in \mathcal{S}$ shall be \emph{weakly mutually exclusive} if and only if 
\begin{equation} a\,\cap\,b \,=\, z \, \& \,p(z)\,=\, 0 \end{equation}
\end{definition}

\begin{theorem}
The dependence function satisfies all of the following properties:

\begin{align}
\tag{Zero1} \delta (x, y) = 0 \longleftrightarrow p(x\cap y) = p(x)p(y)\\
\tag{Zero2} p(x) = 0 \longrightarrow \delta(x,y)= 0\\
\tag{Symmetry} \delta(x, y) = \delta(y, x) \\
\tag{Chaff} x\subset f\, \& \, y \subseteq b \,\&\, x\cap y = f\cap b \longrightarrow \delta(f, b)\leq \delta(x, y) \\ 
\tag{Union} \delta(x\cup y, z) = \delta(x, z) + \delta(y, z) - \delta(x\cap y, z)\\
\tag{Complement1} \delta(x, y) = - \delta(x, y^c)\\
\tag{Complement2} \delta(x, y) = \delta(x^c, y^c)\\
\tag{Identity} \delta(x,x) = p(x) - (p(x))^2 \\
\tag{Unity} \delta(x,\varnothing) = \delta(x,X) = 0\\
\tag{Subset} x\subset y \longrightarrow \delta(x,y) = \dfrac{[1\pm (1-4\delta(x,x))^{\frac{1}{2}}][{1\pm (1-4\delta(y,y))^{\frac{1}{2}}}]}{4}\\
\tag{MEx} x\cap y = \varnothing \longrightarrow \delta(x, y) = -\dfrac{[1\pm (1-4\delta(x,x))^{\frac{1}{2}}][{1\pm (1-4\delta(y,y))^{\frac{1}{2}}}]}{4}
\end{align}
\end{theorem}
\begin{proof}
\begin{itemize}
\item {The proofs of \textsf{Chaff, Identity, Subset} are not in \cite{bd2010}.}
\item {Under the conditions of the premise in \textsf{Chaff}, $p(x\cap y) = p(f\cap b)$, $p(x) \leq p(f)$ and $p(y)\leq p(b)$.}
\item {So $p(x)\cdot p(y) \leq p(f)\cdot p(b)$ and $\delta(f,b) \leq \delta(x, y)$.}
\end{itemize}

\begin{itemize}
\item {\textsf{Identity} follows by a simple substitution of $x$ for $y$ in the definition of $\delta(x, x)$.}
\item {The formula for \textsf{Identity}, yields \[p(x) = \dfrac{1 \pm (1-4\delta(x,x))^{\frac{1}{2}}}{2},\] }
\item {If $x\subset y $, then $\delta(x, y) = p(x)p(y^c)$. Substitution of the expression for $p(x)$ in terms of $\delta(x, x)$ yields the required formula for \textsf{Subset}.}
\end{itemize}
\end{proof}

Mutual inclusivity (as opposed to exclusivity) does not relate easily to concepts of rough dependence, because the concepts are not comparable even if the meta aspects are ignored.  Rough dependence unlike probabilistic dependence is not oriented as no concept of an event being favorable or unfavorable for another event in a negative sense are possible/known to be useful. 

The relation based abstraction of probabilistic dependence developed by the present author in \cite{am3930,am9501,am6000} is as follows: 

\begin{definition}
In the probability space $X$, let 
\begin{itemize}
 \item {$(\forall x, y\in \mathcal{S})$ $\pi x y$ if and only if $p(x) \cdot p(y)\, <\, p(x\cap y)  $}
 \item {$(\forall x, y\in \mathcal{S})$ $\sigma x y $ if and only if $p(x\cap y)\,<\, p(x) \cdot p(y) $}
\end{itemize}
\end{definition}

The intended meaning of $\pi x y$ is \emph{$x$ and $y$ have positive  dependence}, while that of $\sigma x y$ is  
\emph{$x$ and $y$ have negative dependence }. 

\begin{proposition}
\[(\forall \varnothing \subset x, y \subset X)\, \pi  xy \text{ or } \sigma xy \]
\end{proposition}

\begin{theorem}
All of the following hold in a probability space, with $\varnothing \subset x, y , z, a \subset X$ 
\begin{align}
\tag{Identity}  \pi xx \\
\tag{Co-Identity}  (\pi xx \longrightarrow \sigma x x^c )  \\
\tag{Sum1} (\pi x (z\cup y) \longrightarrow \pi xz \text{ or } \pi xy) \\
\tag{Mutual Complementation} (\pi x y^{c} \,\leftrightarrow \, \sigma y x )\\
\tag{Symmetry} (\pi x y \,\leftrightarrow \, \pi y x)\\
\tag{Incompatibility} (\pi xy \,\&\, y\subset a \nrightarrow \pi xa ) \\
\tag{NonExtensionality]} (b\nsubseteq x \longrightarrow (\exists c) \pi b c \,\&\, \neg \pi c x  )\\
\tag{Sum2} (x\cap y \neq \varnothing  \,\longrightarrow\, (\pi x a\,\&\, \pi y a \,\longrightarrow\, \pi (x\cup y) a ))\\
\tag{Co-Sum} (x\cap y \neq \varnothing  \,\longrightarrow\, (\sigma x a\,\&\, \sigma y a \,\longrightarrow\, \sigma (x\cup y) a ))\\
\tag{Set coherence-1} (\varnothing\,\neq\,x\subseteq y \, \longrightarrow\,\pi x y )\\
\tag{Set coherence-2} (x \cap y\,=\, \varnothing  \,\longrightarrow\, \sigma x y)
\end{align}
\end{theorem}
It is possible to specify positive, negative and neutral regions of a subset in general rough sets. But these do not correspond to the situation in probability spaces. 

\begin{definition}
The \emph{common dependence spectra} of any two elements $a, b \in \mathcal{S}$ shall be the set 
\begin{equation}
  U(a, b) = \{x \,;\, \pi ax \,\&\, \pi bx\}. 
\end{equation}

A subset $K \subseteq \wp (\mathcal{S})$ will be a said to be a $\pi-ideal$ if and only if the following two conditions hold:
\begin{align}
(\forall z\in K)(\forall x\in \mathcal{S})\,(\pi zx \longrightarrow x\in K)\\
(\forall z,b\in K)\,U(z, b) \cap K \neq \varnothing 
\end{align}

\end{definition}

Since $\pi$ is symmetric, the dual notion of $\pi-filter$ would coincide with that of $\pi$-ideal. So the following definition makes sense 

\begin{definition}
The \emph{Common Co-Dependence Spectra} of any two elements $a, b \in \mathcal{S}$ shall be the set 
\begin{equation}
  L(a, b) = \{x \,;\, \sigma xa \,\&\, \sigma xb\}. 
\end{equation}
A subset $F \subseteq \wp (\mathcal{S})$ will be a called a $\sigma$-filter if and only if the following two conditions hold:
\begin{align}
 (\forall z\in F)(\forall x\in \mathcal{S})\,(\sigma xz \longrightarrow x\in F)\\
 (\forall z,b\in F)\,  L(z, b) \cap F \neq \varnothing 
\end{align}
\end{definition}

\begin{theorem}
All of the following hold:
\begin{itemize}
\item {The set of all $\pi$-ideals $\mathcal{I}(\mathcal{S})$ is a Bounded Poset ordered by set inclusion.}
\item {The set of all $\sigma$-filters $\mathcal{F}(\mathcal{S})$ is a Bounded Poset ordered by set inclusion.}
\item {Each $\pi$-ideal $K$ is convex in the sense if $\pi xz \,\&\, \pi zy \,\&\, \pi xy$ and $x, y\in K$, then $z\in K$.}
\end{itemize}
\end{theorem}

\begin{theorem}\label{supre}
\begin{equation}
(\forall x, z\in \mathcal{S})(\exists!^{\geq 1} c\in U(x,z))(h\in U(x,z) \longrightarrow h=c \text{ or } \pi ch) 
\end{equation}
That is pairs of sets in the $\sigma$- algebra have at least one $\pi$-supremum.
\end{theorem}

\begin{proof}
\begin{align*}
\text{By definition } U(x,z) = \{a ;\, \pi ax \, \& \, \pi az \},\\
\text{So if } h\in U(x, z) \text{ and}\\
p(x)\cdot p(h) \leq  p(h\cap x) \text{ and} \\
p(z)\cdot p(h) \leq  p(h\cap z) \\
\text{Clearly } x\cup z \in U(x, z) \\
p((x\cup z)\cap h) - p((x\cup z))\cdot p(h) = \\
\left[p(x\cap h) - p(h)\cdot p(x)\right] + \left [ p(x\cap h) - p(h)\cdot p(x)\right] -\\
- \left [p(x\cap z \cap h) - p(x\cap z) \cdot p(h) \right ] \\
 \end{align*}

But the third summand is less than either of the first two positive summands.
So $p((x\cup z)\cap h) - p((x\cup z))\cdot p(h) \geq 0$.

This proves that an admissible value of $c$ is $x\cup z$. 
\end{proof}

\begin{theorem}
The supremum need not be unique in Theorem \ref{supre} and if $K$ is a $\pi$-ideal and $x, z\in K$, then the supremum(s) $c = \sup (x, y)\in K$.   
\end{theorem}

\begin{proof}

Suppose $c\notin K$, then there exists $b\in K\cap U(x, z)$. 

But this would yield $\pi c b$, 

which in turn contradicts $c\notin K$ as $K$ is a $\pi$-ideal.

\end{proof}

Collecting the non-unique supremums can be useful for generating $\pi$-ideals through global operations.

\begin{definition}
The following global operations collect neighborhoods and supremums associated with subsets of the sigma algebra:
\begin{equation*}
\mathfrak{L}, \lambda : \wp (\mathcal{S})\setminus \varnothing) \longmapsto \wp (\mathcal{S}) \tag{Neighborhood Maps} 
\end{equation*}
\begin{equation*}
(\forall Q\in \wp (\mathcal{S})\setminus \varnothing)) \, \mathfrak{L}(Q) = \{x ; (\exists b\in \mathcal{S}) \, \pi xb\}, \; \lambda(Q) = Q\cup \mathfrak{L}(Q)  
\end{equation*}
\begin{equation*}
\mathfrak{Z}, \Xi : \wp (\mathcal{S})\setminus \varnothing) \longmapsto \wp (\mathcal{S}) \tag{Global Supremums}
\end{equation*}
\begin{equation*}
(\forall Q\in \wp (\mathcal{S})\setminus \varnothing)) \, \mathfrak{Z}(Q) = \{x ; (\exists a, b\in Q) \, x = Sup(a, b)\} \,\&\,\Xi(Q) = Q\cup \mathfrak{Z}(Q) 
\end{equation*}
\end{definition}

\begin{theorem}
Since $\Xi, \lambda, \mathfrak{Z, \, L}$ are operations, they can be composed and applied recursively. Thus $(\Xi\lambda)^n$ is an abbreviation for 
\[\stackrel{\underbrace{(\Xi\circ\lambda)\circ(\Xi\circ\lambda)\circ\ldots(\Xi\circ\lambda)}}{\text{n times}}\]
If $<B>$ is the $\pi$-ideal generated by a nonempty subset of $X$, then 
\begin{equation}
<B> = \bigcup_{1}^{\infty} (\Xi\lambda)^n (B) = \bigcup_{1}^{\infty} (\mathfrak{ZL})^n (B) 
\end{equation}
\end{theorem}

\begin{proof}
The proof basically consists in 
\begin{itemize}
\item {Direct verification of the fact that $\bigcup_{1}^{\infty} (\Xi\lambda)^n (B)$ is a $\pi$-ideal.}
\item {Verification of the assertion that $\bigcup_{1}^{\infty} (\Xi\lambda)^n (B)$ is the smallest $\pi$-ideal containing $B$ through a contradiction argument.}
\item {Reflexivity and symmetry imply the equality \[\bigcup_{1}^{\infty} (\Xi\lambda)^n (B) = \bigcup_{1}^{\infty} (\mathfrak{ZL})^n (B).\]}
\end{itemize}
\qed
\end{proof}

\subsection{Dependence Based Deviant Probability}

The dependence based axiomatic approach to probability considered in \cite{am3930,am9501,am6000} was radically extended for comparison with ideas of rough dependence in the research paper \cite{am9411} by the present author. In deviant probability specialized algebraic models of a new non Bayesian epistemological probability are used without the excesses of numeric valuation. This permits comparison of dependence from an apparently equivalent set theoretic footing.   Results on the structure of the model are proved, dependence in the model is compared with concepts of rough dependence and the meaning of the assumptions are also discussed.

In the axiomatic approach to dependence based probability in \cite{am3930}, the predicates $\pi, \sigma $ concern positive and negative probabilist dependence. The dependence function on the other hand is real valued. The main motivation for introducing a function that takes value in the $\sigma$-algebra were:
\begin{itemize}
\item {A purely set-based probabilistic dependence theory would be a possibility - this would be relevant for dealing with statistical information from mixed sources for example,}
\item {Probabilistic dependence is about causality, }
\item {Probability is essentially about relations on sets - numeric values can be deceptive (especially when axioms are relaxed),}
\item {Comparison with rough dependence would be easier and}
\item {a new internalized generalized probability theory that avoids measures. }
\end{itemize}

\begin{definition}\label{pi}
The \emph{positive deviance} of $x\in \mathcal{S}$ relative $y\in \mathcal{S}$ will be the value of the function $\pi_o$ defined by 
\begin{equation}
\pi_o (x,y) = \chi \left[\max \{z \,; \, p(z) \leq p(x\cap y) - p(x)\cdot p(y) \,\&\, z\in \mathcal{S}\,\&\,z \subset x\cap y\}\right ],                                                                                                                                                                                                                                                                       \end{equation}
with $\chi$ being a choice function $:\wp (\mathcal{S})\longmapsto \mathcal{S}$.

Analogously, the \emph{negative deviance} of $x\in \mathcal{S}$ relative $y\in \mathcal{S}$ will be the value of the function $\sigma_o$ defined by 
\begin{equation}
\sigma_o (x,y) = \xi \left[\min \{z : \, p(x)\cdot p(y) - p(x\cap y) \leq p(z)\,  z\in \mathcal{S}\,\&\,z \subset (x\cap y)^c \}\right ],                                                                                                                                                                                                                                                                                                         \end{equation}
with $\xi$ being a choice function $:\wp (\mathcal{S})\longmapsto \mathcal{S}$.
\end{definition}

The concept of conditional events naturally carries over with its temporal import, but concepts of expectation can be generalized in multiple ways. A set theoretic way of arriving at expectations using deviances and set-theoretic operations is possible.

\begin{definition}\label{deviant}
$x$ will be said to be \emph{deviant equivalent} to $y$ (in symbols $x \approx y$ ) if and only if 
\begin{equation}
\pi_o(\pi_o (x, y),x) = \pi_o(\pi_o(x, y), y) \,\&\, \sigma_o(\sigma_o (x, y),x) = \sigma_o(\sigma_o(x, y), y)  
\end{equation}
 \end{definition}

The functions have many nice properties and generate a number of interesting properties:

\begin{proposition}
In the above context,
\begin{align}
\pi_o (x, y\neq z \,\&\, p(z)=0 \longrightarrow \sigma_o(x, y) =g \,\&\, 0 < p(g)  \\
\sigma_o (x, y)\neq z \,\&\, p(z)=0 \longrightarrow \pi_o(x, y) =g \,\&\, 0 < p(g).
\end{align}
\end{proposition}

\begin{theorem}
All of the following hold:
\begin{align}
\tag{Symmetry} \pi_o (x,y) = \pi_o (y, x)\\
\tag{Bottom} \pi_o (x, \varnothing) \approx \varnothing \\
\tag{Top} \pi_o (x, X) \approx \varnothing \\
\tag{Almost Empty} \pi_o(x, x) \approx \varnothing \longrightarrow p(x)=0 \\
\tag{Non-associativity} \neg \left(\pi_o (x, \pi_o(y, z)) \approx \pi_o (\pi_o (x, y), z)\right)\\
\tag{Identity2} x\approx x \\
\end{align}
\end{theorem}

\begin{proof}
\begin{itemize}
\item{Symmetry follows from the definition of the choice function as its argument would be the same for $\pi_o (x,y)$ and $ \pi_o (y, x)$.}
\item{$\pi_o (x, \varnothing) = z$ implies $p(z) \leq p(x\cap \varnothing ) - p(x)\cdot p(\varnothing) = 0 $. So $z \approx \varnothing$. }
\item{If $\pi_o (x, X) = z$, then $p(z) = p(x\cap X) - p(x)\cdot p(X) = 1 -1 = 0$. So $z \approx \varnothing$. }
\item{Suppose $x\neq \varnothing$, then $\pi_o(x, x) =z$ implies $p(z) = p(x) - (p(x))^2= p(x)p(x^c)$. But $p(z) = 0 $ yields either $p(x) = 0$ or $p(x^c) = 0$. So $z \approx \varnothing$. }
\item{Non-associativity happens because the choice function does not impose any constraints on the process of generation of maximal sets. So in general $\neg \left(\pi_o (x, \pi_o(y, z)) \approx \pi_o (\pi_o (x, y), z)\right)$.}
\item{Substituting $x$ for all variables in the definition of $\approx $, the assertion can be verified. }
\end{itemize}
\end{proof}

Proceeding along similar lines, the following theorem can be proved:

\begin{theorem}
All of the following hold:
\begin{align}
\tag{S-Symmetry} \sigma_o (x,y) = \sigma_o (y, x)\\
\tag{S-Bottom} \sigma_o (x, \varnothing) \approx \varnothing \\
\tag{S-Top} \sigma_o (x, X) \approx \varnothing \\
\tag{S-Almost Empty} \sigma_o(x, x) \approx \varnothing \longrightarrow p(x)=0 \\
\tag{S-Non-associativity} \neg \left(\sigma_o (x, \sigma_o(y, z)) \approx \sigma_o (\sigma_o (x, y), z)\right)\\
\tag{Domain} dom(\pi_o) \cup dom(\sigma_o)\subset \mathcal{S}^2
\end{align}
\end{theorem}

Below certain important sequences and substructures related to dependence are investigated.

\begin{definition}
In the context of definition \ref{pi}, the following sequence of functions can be recursively defined and will be referred to as the \emph{Dependence Trail}($\mathfrak{Tr} (x,z)$) of $x$ on $z$:
\begin{align*}
\pi_1(x,z) = \pi_o(\pi_o(x, z).x) \tag{2nd Step}\\
\ldots \ldots \\
\pi_r(x, z) = \pi_o (\pi_{r-1}(x, z), x) \tag{rth Step}\\
\pi_n (x, z) = \pi_{n+1}(x, z) \tag{Stopping Criteria}
\end{align*}
\[\mathfrak{Tr}(x, z) = \{\pi_o (x,z), \pi_1 (x, z), \ldots , \pi_n (x, z), \ldots \}\]

If the stopping criteria is attained at $n$ then the \emph{length of dependence} of $\mathfrak{Tr}(x, z)$ will be $n$ (but all sequences shall be interpreted as infinite ones). The collection of all dependence trails on a probability space $S$ will be denoted by  $\mathfrak{Tr}(S)$. 
\end{definition}
\begin{proposition}
If $\mathfrak{Tr}(x, z) = \{\pi_o (x,z), \pi_1 (x, z), \ldots , \pi_n (x, z), \ldots \}$ is a dependence trail, then $p(\pi_r (x, z)) \leq p(\pi_{r-1} (x, z)) $ for all $r$. 
\end{proposition}

The above proposition essentially says that sets in a $\sigma$-algebra are strongly self-doubting about their own dependencies.

\begin{proposition}
If $a, b \in \mathfrak{Tr}(S)$ (with $a = \{a_j\}$ and $b= \{b_j \}$), let \[ a \prec b \text{ if and only if } \bigwedge_j  p(a_j) \leq p(b_j),\]  then $\prec$ is a quasi order on $\mathfrak{Tr}(S)$.
\end{proposition}

\begin{proof}
\begin{itemize}
\item {Reflexivity of $\prec$ is clear.}
\item {If $a\prec b $ and $b\prec e= \{e_j \}$ then for each $j$, $a_j \leq b_j \leq c_j$, so $a\prec c$ and transitivity holds } 
\item {Antisymmetry is obviously false in general as $p(x) = p(z)$ does not imply $x = z$.}
\end{itemize} 
\end{proof}

\subsection{Comparison of Dependence}

The problem of comparison of rough and probabilist dependence is not an easy one because of the following reasons:
\begin{itemize}
\item {The concepts themselves are very plural things at the model theoretic level. }
\item {The ontologies of each of the components of the plurals have features that are relatively unique that the justifications for comparison become very suspect.}
\item {If the comparison is to be at the level of solving practical problems, then the methods used are hybrid ones and a host of hidden assumptions become apparent on scrutiny.}
\end{itemize}
Because of this the best way is to look at minimal parts that make semantic and ontological sense in a \emph{suitable common framework} without involving numeric valuations. 

One version of such a perspective is the set version of rough and probabilist dependence functions. From what has been said in this section, if the comparison is between decontaminated classical rough sets and set-valued dependence based probability, then it is as follows:

\begin{theorem}
The properties common to $\pi_o$ and $\beta_i$ are Symmetry, Bottom and Almost Empty. Other properties are not shared.  
\end{theorem}
But there is no concept corresponding to the $\sigma$ predicates and functions.

\section{Tarski Algebras and CAS}

Tarski algebras are the same thing as implication algebras \cite{rh}. A few full dualities relating to classes of such algebras are known. Two related dualities are outlined in this section. One of this is a duality for finite Tarski sets \cite{csa,sclc2008} or covering approximation spaces. An adaptation to rough contexts can be found in this research chapter \cite{am5019} by the present author.

The duality between Boolean algebras and Boolean spaces is an example of a topological duality - this basic result can be generalized to a duality between Tarski algebras and spaces\cite{mabad2004} . Full dualities between the category of Boolean algebras with meet-morphisms that preserve $1$ and Boolean spaces with Boolean relations are also known \cite{hal1962}.  This has been generalized \cite{sclc2008} to Tarski algebras.

\begin{definition}
A \emph{Tarski algebra} (or an \emph{implication algebra} IA) is an algebra of the form $S= \left\langle\underline{S}, \cdot, 1   \right\rangle$ of type $2, 0$  that satisfies (in the following, the implication $a\cdot b$ is written as $ab$ as in \cite{rh}. Further brackets must be added from the left unless indicated otherwise.)
\begin{align*}
1 a = a \tag{T1}\\
aa = 1 \tag{T2}\\
a(bc) = (ab)(ac) \tag{T3}\\
(ab)b = (ba)a \tag{T4}
\end{align*}

\end{definition}

The variety of IAs is denoted by $\mathcal{V}_{IA}$. If $X$ is a set, and $(\forall A, B\in \wp(S))\, A\cdot B = A^c\cup B$, then $\left\langle\underline{\wp(X)}, \cdot, X   \right\rangle$ is an IA. Any subalgebra of such an algebra is said to be an \emph{IA or Tarski algebra of sets}. A join-semilattice order $\leq$ is definable in a Tarski algebra as below:
\[(\forall a, b)\, a \leq b \leftrightarrow ab=1 ; \text{ the join is } a\vee b = (ab)b\]

Note that any partially ordered set $Q= \left\langle\underline{Q}, \trianglelefteq  \right\rangle$ can be transformed into a groupoid by defining a binary $\odot$ as follows:
\[a\odot b \,=\, \left\{
\begin{array}{ll}
a,  & \mathrm{if}\,\, a\trianglelefteq b \\
b, & \mathrm{otherwise}
\end{array}
\right.\]
Omitting the operation symbol and binding to the left, it is known that the groupoid is defined by the system
\begin{align*}
aa=a   \tag{O1}\\     
aba = ba   \tag{O2}\\ 
abb =ab   \tag{O3}\\ 
a(abc) = a(bc)   \tag{O4}\\ 
abcb = acb   \tag{O5}
\end{align*}

\begin{proposition}
The following equations are satisfied by both Tarski algebras and poset-groupoids:
\begin{align*}
aaa = a   \tag{TO1}\\ 
a(abc) = a(bc)   \tag{TO2}\\
a(ab) =ab   \tag{TO3}\\
\end{align*}
\end{proposition}

\emph{Filters or deductive systems} of an IA $S$ are subsets $K\subseteq S$ that satisfy 
\[1\in K \, \&\, (\forall a, b)(a, ab\in K \longrightarrow b\in K)\]
The set of all filters $\mathcal{F}(S)$ is an algebraic, distributive lattice whose compact elements are all those filters generated by finite subsets of $S$. A filter $K$ is prime if and only if it satisfies $(\forall a, b)(a\vee b\in K \longrightarrow a\in K \text{ or } b\in K)$. 

\begin{theorem}
In a finite Tarski algebra $S$, the following hold:
\begin{itemize}
\item {A filter is prime if and only if it is a maximal filter.}
\item {A filter is prime or maximal iff it is of the form $(x\downarrow)^c$ for a coatom $x$}
\item {If $Spec(S)$ is the set of prime or maximal filters of $S$ and $\sigma_S: S\longmapsto \wp(Spec(S))$ is a map into the Tarski algebra of sets $\wp(Spec(S))$ and is defined by \[(\forall x )\, \sigma_S(x) = \{K:\, x\in K\in Spec(S)\},\] then $\sigma_S$ is an embedding }
\end{itemize}
\end{theorem}

In the last theorem if $S$ is a finite Boolean algebra, then it is provable that $Spec(S)\cong Spec(\wp(Spec(S))$ and in fact for any finite Boolean algebra $S$, $S\cong \wp(Spec(S))$. This does not hold for finite IA. But note that $Spec(S)$ is determined by the set $CoAt(S)$ of coatoms. 

\begin{definition}
A \emph{Tarski set} is a pair $\left\langle X, \mathcal{S}  \right\rangle$ where $X$ is a non-empty set and $\mathcal{S}$ is a nonempty subset of $\wp(X)$. It is \emph{dense} (or a \emph{covering approximation space} (CAS)) if and only if $\bigcup(\mathcal{S}) = X $. The \emph{dual} of a Tarski set $\left\langle X, \mathcal{S}  \right\rangle$ is the subset $\Delta(X)\subset \wp (X)$ defined as below: \[\Delta (X) = \{U : \, (\exists W\in \mathcal{S})(\exists H\subseteq W)\, U= W^c\cup H \}\]
\end{definition}

The dual of a covering approximation space is simply an implication algebra of sets. 

\begin{theorem}
Let $\left\langle X, \mathcal{S}  \right\rangle$ is a Tarski set, then $\left\langle\Delta (X), \cdot, X  \right\rangle$ is a Tarski subalgebra of sets.
\end{theorem}

The proof is by direct verification.

If $S$ is a finite Tarski algebra and $\sigma_S: S\,\longmapsto \wp(Spec(S))$ is the map defined earlier and $\mathcal{K}_S = \{\sigma(x)^c:\, x\in S \}$, then the Tarski set $\left\langle Spec(S), \mathcal{K}_S   \right\rangle$ is also referred to as the \emph{associated set} of $S$.

\begin{theorem}
If $S$ is a finite Tarski algebra, then $\sigma_S(S) = \Delta(Spec(S))$ and so $S\cong \Delta(Spec(S))$.
\end{theorem}
\begin{proof}
\begin{itemize}
\item {Since $(\forall x)\, \sigma_S(x) = \sigma(x)\cup \emptyset$, therefore $\sigma_S(x)\in \Delta (Spec(S))$.}
\item {Let $U\in \Delta(Spec(S))$.  By definition, $(\exists x\in S)(\exists H\subseteq (\sigma_S(x))^c)\, U= \sigma_S(x)\cup H$}
\item {Let $H= \{Q_1, \ldots , Q_n \}$ For each of these maximal filters $Q_i$, there exists a coatom $q_i$ that generates it. }
\item {So $\sigma_S(q_i)^c = \{(q_i\downarrow)^c\}$. So $H= \bigcup (\sigma_S(q_i))^c$.}
\item {This means $U= \sigma(x) \cup \bigcup (\sigma_S(q_i))^c = \sigma_S(b)$ for some $b$}
\item {So $U\in \sigma_S(S)$ and $\sigma_S(S) = \Delta(Spec (S))$.}
\end{itemize}
\end{proof}

\begin{theorem}
Let $\left\langle X, \mathcal{S}  \right\rangle$ be a finite dense Tarski set or a CAS, then the map $\xi_X: X\longmapsto Spec(\Delta(X))$ defined by $\xi_X(x) = \{U:\, x\in U\in \Delta(X)\}$ is injective and a surjection.
\end{theorem}

\begin{proof}
The proof is by direct verification.
\begin{itemize}
\item {If $a, b\in X$ are distinct elements, then $b\in \{a\}^c\in Coat(\Delta(X))$. So $\xi_X$ is injective.}
\item {By finiteness, $(\forall Q\in Spec(\Delta(X)))(\exists U\in Coat(\Delta(X)))\, Q= (U\downarrow)^c$.}
\item {For a specific $Q$ and $U$ in the last statement, $(\exists x\in X)\,  U= \{x\}^c$ as $\left\langle X, \mathcal{S}  \right\rangle$ is dense. Clearly then $\xi_X(x) = Q$ and $X\cong Spec (\Delta (X))$. }
\end{itemize}
\end{proof}

The result is an abstract representation theorem for finite Tarski algebras. The actual significance of the result has not been explored in the context of covering approximation spaces (even in the finite case) before \cite{am5019}. More aspects are proved in this paper in the following subsection. For one thing, every construct in a CAS has an algebraic representation. 

For extending the results to the infinite case, a topological extension is necessary.

\begin{definition}
A \emph{Tarski space} (T-space) is a concrete topological structure of the form $\chi = \left\langle X, \mathcal{K}, \tau    \right\rangle$ that satisfies:
\begin{enumerate}
\item {$\left\langle X, \tau  \right\rangle$ is a Hausdorff, totally disconnected topological space with $\mathcal{K}$ being a basis for the compact subsets of $\tau$.}
\item {$(\forall A, B\in \mathcal{K})\, A\cap B^c\in \mathcal{K}$}
\item {For any two distinct $a, b\in X$, exists a $U\in \mathcal{K}$ such that $a\in U$ and $b\notin U$.}
\item {If $F$ is a closed subset and $\{U_i\}_{i\in I}$ is a directed subcollection of sets in $\mathcal{K}$ and for each $i\in I$, $F\cap U_i \neq \emptyset$, then $F\cap (\bigcap U_i ) \neq \emptyset$. }
\end{enumerate}
\end{definition}

Given a T-space two distinct Tarski subalgebras of a set Tarski algebra are defined in \cite{sclc2008}:
\begin{align*}
T_{\mathcal{K}}(X) = \{W^c \cup H: H\subseteq  W \in \mathcal{K}\}   \tag{T-algebra}\\ 
\Delta_{\mathcal{K}}(X) = \{U:\, U^c\in \mathcal{K} \}   \tag{dual T-algebra}   
\end{align*}

\begin{theorem}\label{triv}
\begin{itemize}
\item {If $X\in \mathcal{K}$, then $\chi$ is a Boolean space and $\Delta_{\mathcal{K}}(X)$ is a Boolean algebra of all clopen sets of the topological space.}
\item {If $X$ is finite, then $T_{\mathcal{K}}(X) = \Delta_{\mathcal{K}}(X)$ }
\end{itemize} 
\end{theorem}

If $A, B\in \mathcal{V}_{IA}$, then a \emph{semi-morphism} is a monotone map $f: A\longmapsto B $ that satisfies 
\begin{itemize}
\item {$f(ab) \leq f(a)f(b)$}
\item {$f(1) = 1$}
\end{itemize}

\begin{example}
If $X$ and $ W$ are sets and $R\subset X\times W$, let $[x]_i = \{a: Rxa \}$. Define a map $h_r: \wp(W) \longmapsto \wp (X)$ such that for any $U\subseteq W$, \[h_R(U) = \{x: [x]_i\subseteq U\}\] 
$h_R \in SMor(\wp(W),\wp(X))$ - the set of semi-morphisms $:W\longmapsto X$.
\end{example}

\begin{definition}
Let $\chi_X$ and $\chi_W$ be two T-spaces over $X$ and $W$ respectively, then $R\subseteq X\times W$ is a \emph{T-relation}
if and only if the following hold:
\begin{itemize}
\item {$(\forall U\in \Delta_{\mathcal{K_W}}(W))\,h_R(U) = \{x: [x]_i\subseteq U\}\in \Delta_{\mathcal{K}_X}(X) $}
\item {$[x]_i$ is a closed subset of $W$ for each $x\in X$}
\end{itemize}
A \emph{T-partial} function is a partial map $f: X\longmapsto W$ such that for each $U\in \Delta_{\mathcal{K_W}}(W)$, $f^{-1}(U)\in \Delta_{\mathcal{K_X}}(X)$. The set of all T-partial functions (resp. relations) from $X$ to $W$ will be denoted by $TF(X, W)$ (resp. $TR(X, W)$).
\end{definition}

\begin{definition}
The following categories can be defined on the basis of the above:
\begin{itemize}
\item {$\mathfrak{TR}$ with Objects being Tarski spaces and Morphisms being sets of T-Relations.}
\item {$\mathfrak{TF}$ with Objects being Tarski spaces and Morphisms being sets of T-partial functions.}
\item {$\mathfrak{ST}$ with Objects being Tarski algebras and Morphisms being sets of semi-morphisms.}
\item {$\mathfrak{HT}$ with Objects being Tarski algebras and Morphisms being sets of homomorphisms.}
\end{itemize}
\end{definition}

\begin{theorem}
\begin{itemize}
\item {$\mathfrak{HT}$ is a subcategory of $\mathfrak{ST}$,}
\item {$\mathfrak{ST}$ is dually equivalent to $\mathfrak{TR}$, and}
\item {$\mathfrak{HT}$ is dually equivalent to $\mathfrak{TF}$.}
\end{itemize}
\end{theorem}

For the long proof, the reader is referred to \cite{sclc2008}.

\section{Covering Approximation Spaces and Topology}

As mentioned in the first section a tuple $\mathfrak{C} =\left\langle\underline{S},\, \mathcal{C}   \right\rangle$ is said to be a \emph{covering approximation space} when $\underline{S}$ is a set and $\mathcal{C}$ is a cover on it. Typically, the set is a collection of attributes and the basic operations that are performed on it to generate approximations \cite{am240,yy2012c} suggest the following:
\begin{itemize}
\item {attributes may be aggregated in any way,}
\item {at least a bounded distributive lattice structure (relative to set union and intersection) on the structure generated by it is assumed,}
\item {complementation is also assumed often on the generated structure, }
\item {$\mathcal{C}$ may not be a granulation and }
\item {it may used for defining both granular and non-granular approximations.}
\end{itemize}

So it may appear that the Boolean algebra or the bounded distributive lattice generated by $\mathcal{C}$ is used in forming approximations. The following practical example shows that such an idealized approach can be unjustified.
\begin{example}
In the incomplete information table relating to medical diagnostics of patients, suppose that the columns labeled as \emph{Dress} and \emph{Color} correspond to responses to color of dress worn and favorite color of the patient dresses while the column labeled as \emph{State} corresponds to diagnosis and other columns correspond to symptoms of the patient. It is possible that superfluous associations may be generated by the table if all attributes are taken into account. In big data situations such events may not be easily tractable and meta methods that can handle them in a dynamic way are of much interest. Partial aggregations and commonalities have the potential to form a basic framework for handling the issue.

\begin{table}[h]
\centering
\begin{tabular}{|c|c|c|c|c|c|c|c|}
\hline
Nom & Temp. & Body Pain & Skin & H.ache & Dress &  Color & State\\
\hline
\textbf{A} & Set1  & Medium  & $1$ & No & red & red & F0\\

\textbf{B} & Set2 & None & $1$ & Yes & pink & red& F0 \\

\textbf{C} & Set1 & Mild  & NA & No & purple & pink & Test \\

\textbf{E} & Medium  & Medium  & $1$ & Yes &NA & white & F1\\

\textbf{F} & High & None & $1$ & Yes & white & NA &Test  \\

\textbf{G} & Set1 & High  & NA & Yes & F1 \\

\hline
\end{tabular}
\caption{Medical Diagnostics Data}
\end{table}

\end{example}

Thus in many application contexts, partial Boolean algebras and partial distributive lattices are/should be used for the purpose. Then again the semantics of rough objects of interest in the context can be quite different. All three approaches may be compared with probabilist or possibilist approaches. It will help if these are named for convenience:

\begin{itemize}
\item {\textsf{Scheme-S:} In this one uses the algebraic semantics associated with the rough objects.  } 
\item {\textsf{Scheme-B:} In this one uses the Boolean algebraic semantics associated with the all objects.}
\item {\textsf{Scheme-P:} In this one uses a partial Boolean algebraic semantics associated with the all objects.}
\end{itemize}

Relative to the duality considered in this paper the three approaches lead to diverse results. 

\subsection{Topological Operators}

Some other related concepts are   

If every element $K$ of a cover $\mathcal{S}$ contains an element $x\in S$ that satisfies 
\[(\forall Z\in \mathcal{S})\, (x\in Z \longrightarrow K\subseteq Z)\] then $\mathcal{S}$ and $x$ are said to be a \emph{representative} cover and element respectively. 

A covering $\mathcal{S}$ is said to be \emph{unary} if and only if $(\forall x\in S)\,\#(md(x))=1$. This condition is equivalent to $(\forall K_1, K_2\in \mathcal{S})(\exists C_1,\ldots C_n \in \mathcal{S})\, K_1 \cap K_2\, =\,\cup_1^n C_i $. For a proof, see \cite{zw3} and \cite{am5019} by the present author.

\begin{definition}\label{subsystema}
The \emph{intersection closure} of $\mathcal{S}$, is denoted by $Cl_\cap (\mathcal{S})$, is the least subset of $\wp (S)$ that contains $\mathcal{S}, \, S, \emptyset $, and is closed under set intersection. The \emph{union closure} of $\mathcal{S}$, is denoted by $Cl_\cup (\mathcal{S})$, is the least subset of $\wp (S)$ that contains $\mathcal{S}, \, S, \emptyset$, and is closed under set union. More generally if $\underline{\mathcal{H}}$ and $\overline{\mathcal{H}}$ are dual closure and closure systems contained in $\wp(S)$,
\end{definition}

The following theorem was proved in \cite{zw3}

\begin{theorem}\label{unaryrep}
When $S$ is finite, a covering $\mathcal{C}$ is {unary} if and only if  \[(\forall K_1, K_2\in \mathcal{C})(\exists C_1,\ldots C_n \in \mathcal{C})\, K_1 \cap K_2 = \cup_1^n C_i \]
\end{theorem}

\begin{proposition}
If $L: \wp(S) \mapsto \wp(S) $ is an abstract operator on a set $S$ that satisfies 
contraction, idempotency, monotonicity and top  then there exists a covering $\mathcal{C}$ of $S$ such that the lower approximation $l1$ generated by $\mathcal{S}$ coincides with $L$.
\end{proposition}

\begin{theorem}
For every interior operator $L: \wp(S) \mapsto \wp(S)$ there exists a unary covering $\mathcal{C}$ on $S$ such that the lower approximation of the first type $l1$ generated by $\mathcal{C}$ coincides with $L$.
\end{theorem}
\begin{proof}
Since $L$ is an interior operator, it satisfies top, contraction, monotonicity and idempotence.
By the previous lemma, there exists a cover $\mathcal{C}$ such that the lower approximation of the first type $l1$ generated by it coincides with $L$. 

$L$ satisfies multiplicativity, $(\forall A, B)\, L(A\cap B) = L(A)\cap L(B)$, and this together with Thm.\ref{unaryrep} yields the result.
\end{proof}

In general, the first, second, third and fourth type of upper approximation operators determined by a cover $\mathcal{S}$ on a set $S$  are not topological closure operators. These are defined as below: 
\begin{itemize}
\item {$X^{l1} = \bigcup \{K: \, K\in \mathcal{S}\, \&\, K\subseteq X\}$}
\item {$X^{u1+}\,=\,X^{l1}\cup\bigcup\{\mathrm{md}(x):\, x\in X \} $,}
\item {$X^{u1}\,=\,X^{l1}\cup\bigcup\{\mathrm{md}(x):\, x\in X \setminus X^{l1} \} $ \cite{zbb1998},}
\item {$X^{u2+}\,=\,\bigcup \{K:\,K\in\mathcal{S}, K \cap X\neq\emptyset\}  = \bigcup \{Fr(x): \, x\in X\}$,}
\item {$X^{u3+}\,=\,\bigcup \{\mathrm{md}(x):\, x\in X\} $,}
\item {$X^{u4+}\,=\,X^{l1}\cup \{K:\, K\cap (X\setminus X^{l1})\neq \emptyset\} $,}
\end{itemize}

Closely related to $X^{u1}$ is $X^{u1+} = X^{l1}\cup\bigcup\{\mathrm{md}(x):\, x\in X \}$
These have been defined many times over in the literature (see \cite{am501,am240,yy2012c}).

In \cite{wzw2006,wz2009b,wzw2007}, conditions for the upper approximation operators to be closure operators are proved, but the conditions do not amount to the operators being topological closure operators. In \cite{xge2012}, the following is proved:

\begin{theorem}
The following are equivalent if $S$ is finite:
\begin{enumerate}
\item {$\mathcal{C}$ is a unary cover of $S$.}
\item {$\mathcal{C}$ is a base for some topology $\tau$ on $S$}
\item {\[(\forall K_1, K_2 \in \mathcal{C})(\forall x\in K_1\cap K_2)(\exists K\in \mathcal{C})\, x\in K\subseteq K_1\cap K_2\] }
\item {$u1$ is a topological closure operator.}
\end{enumerate}
\end{theorem}
\begin{proof}
The equivalence of the first and third statement will be proved first. 
If $\mathcal{C}$ is unary, let $(\forall A, B \in \mathcal{C})(\forall x\in A\cap B) \text{ md}(x) = \{K_x \} $. $x$ must be a representative element of $K_x$. So $x \in K_x \subseteq  A\cap B$, with $K_x \in \mathcal{C}$.

If $\mathcal{C}$ is not unary,then $(\exists A, B \in \mathcal{C}) A, B \in \text{ md}(x) \,\&\, A\neq B$. But by the third statement, there must exist a $K \subset A\cap B$ satisfying $x\in K$. This contradicts $A, B \in \text{ md}(x)$.  

From the above, it follows that $\mathcal{C}$ is a unary cover if and only if there exists a topology $\tau$ on $S$ such that $\mathcal{C}$ is a base for the topology.
\end{proof}

The following example shows that it is not possible to generalize to the infinite case:

\begin{example}
\begin{itemize}
\item {Let $S= [-1, 1]$}
\item {$\mathcal{C} = \{\{x\}: x\in S\setminus \{0\} \}\cup \{(-\frac{1}{n},\, \frac{1}{n}):\, n\in N\}\,\cup \{\{-1, 0, 1\}\}$}
\item {$\mathcal{C}$ covers $S$ and if $x\neq 0$ then $md(x) = \{\{x\}\}$ and $md(0) = \{\{-1,0, 1\}\}$. $\mathcal{C}$ is unary, but the theorem does not hold.}
\end{itemize}
\end{example}

\begin{example}
\begin{itemize}
\item {Let $S= R$ - the set of reals and $\mathcal{C} = \{(x-\frac{1}{n},\, x+ \frac{1}{n}):\, x\in S\,\&\,n\in N \}$.}
\item {$\mathcal{C}$ is a base for the usual topology on S.}
\item {$(\forall x\in S)\, md(x) = \emptyset \, \&\, \{x\}^{u1}=\emptyset$. So $u1$ is not a closure operator.}
\end{itemize}
\end{example}

\begin{theorem}
When $S$ is a finite or an infinite set, $u2+$ is a topological closure operator if and only if $\{Fr(x): \, x\in S \}$ forms a partition of $S$.
\end{theorem}
\begin{proof}
The converse is obvious. 
\begin{itemize}
\item {Let $u2+$ be a topological closure operator. It suffices to show that \[(\forall a, b)(Fr(a)\neq Fr(b) \longrightarrow Fr(a)\cap Fr(b) = \emptyset).\] Or else there exists $z\in Fr(a)\cap Fr(b)$.}
\item {Clearly, $Fr(z) \subseteq \bigcup \{Fr(x):\, x\in Fr(b) \} = Fr(b)$ and }
\item {$Fr(b) \subseteq \bigcup \{Fr(x):\, x\in Fr(z)\} = Fr(z)$. So $Fr(b) = Fr(z)$}
\item {Similarly, $Fr(a) = Fr(z) = Fr(b)$ }
\item {This contradicts $Fr(a)\neq Fr(b)$.}
\end{itemize}
\end{proof}

\begin{theorem}
$u2+$ is a topological closure operator if and only if 
there is a closed-open topology $\tau$ (that is a union of members of a partition on $S$) on $S$ such that $\{Fr(x):\, x\in S \}$ is a base of $\tau$ if and only 
\[(\forall a, b \in S)\, Fr(a) \cap Fr(b) = \emptyset \text{ or } a\in Fr(b)\]
\end{theorem}

\begin{proof}
The proof is by an extension of the the proof of the previous theorem. 
\end{proof}

Let $cfr(x) = \bigcup md(x)$ for any $x\in S$

\begin{theorem}
For a finite or an infinite $S$, $u3+$ is a topological closure operator if and only if 
 each $x\in S$ is a representative element of $cfr(x)$ for the unary cover $\{cfr(x) :\, x\in S\}$.  
\end{theorem}
\begin{proof}
If $u3+$ is a closure operator then 
\begin{itemize}
\item {for each $a\in S$, if $a\in cfr(b)$ for some $b\in S$, then 
\[cfr(a)\,\subseteq \bigcup \{cfr(z): \, z\in cfr(b) \} = cfr(b)^{u3+} = cfr(b) \]}
\item {So $a$ must be a representative element of $cfr(a)$ for the cover $\{cfr(z) : z\in S\}$}
\end{itemize}
If $x$ is a representative element of $cfr(a)$ for the cover $\mathbb{C} = \{cfr(z): \, z\in S\}$, then 
\begin{itemize}
\item {For each $z\in \{x\}^{u3+} = cfr(x)$, since $z$ is a representative element of $cfr(z)$ for the cover $\mathbb{C}$, 
\[\{z\}^{u3+} = cfr(z) \subseteq cfr(x) = \{x\}^{u3+} \] }
\item {So \[\{x\}^{u3+ u3+} = \bigcup \{cfr(z) :\, z\in cfr(x)\} \subseteq \{x\}^{u3+}. \] }
\item {This verifies idempotence. Rest of properties can be directly checked.}
\end{itemize}
\end{proof}

\begin{theorem}
For a finite or an infinite $S$, $u3+$ is a topological closure operator if and only if 
$\{cfr(x):\, x\in S \}$ is a base for a topology $\tau$ on $S$ and for each $x\in S$, $\{cfr(x)\}$ is a local base at $x$.
\end{theorem}
\begin{proof}
The proof follows from the previous theorem and the result that for every unary cover there exists a topology $\tau$ on $S$ for which $\{cfr(x): \, x\in S\}$ is a base for $(S, \tau )$. The missing steps can be found in \cite{xge2012}.
\end{proof}

For the proof of the next three theorems, the reader is referred to \cite{xge2012}.
\begin{theorem}
For any finite or infinite $S$, $u4+$ is a closure operator if and only if  the cover $\mathcal{C}$ satisfies 
For all $K_1, K_2\in \mathcal{C}$ if $K_1\neq K_2 \,\&\, K_1\cap K_2 \neq \emptyset$ then  $(\forall x\in K_1\cap K_2)\, \{x\}\in \mathcal{C}$.
\end{theorem}
\begin{theorem}
For any finite or infinite $S$, $u4+$ is a closure operator if and only if  the cover $\mathcal{C}$ is a base for a topology $\tau$ on $S$ and 
\begin{itemize}
\item {$(S, \tau)$ is a union of disjoint subspaces $S_1$ and $S_2$. }
\item {For any distinct $A, B\in \mathcal{C}$, either $A\cap S_2 = B\cap S_2 = \emptyset$ or $A\cap S_2 \neq B\cap S_2$ and $\{F\cap   S_2: \, F\in \mathcal{C}\}$ is a partition of $S_2$, and}
\item {In the topologies $\tau_1,\, \tau_2$ induced on $S_1$ and $S_2$ respectively, $S_1$ is a discrete topological space and $S_2$ is a pseudo-discrete space. }
\end{itemize}
\end{theorem}

From the above results it can be deduced that
\begin{theorem}
$u4+ $ is a closure operator $Rightarrow$ $u1 $ is a closure operator $Rightarrow$ $u3+ $ is a closure operator and no other relation between similar statements hold.
\end{theorem}

A number of if and only conditions for a covering $\mathcal{C}$ being unary are known. Some of these are summarized below:

\begin{theorem}
A cover $\mathcal{C}$ of a set $S$ is unary if and only if
\begin{itemize}
\item {$u3+ = u1$}
\item {$(\forall x\in S)\, nbd(x) \in \mathcal{C}$}
\item {$(\forall X\subseteq S)\, (X^{u4+})^{u3+} = X^{u4+}$.}
\item {$(\forall X\subseteq S)\, (X^{u2+})^{u3+} = X^{u2+}$.}
\end{itemize}
\end{theorem}

\section{Correspondences from Rough Perspective}

From the previous section, it should be clear that in the finite case the following summary holds.

\begin{itemize}
\item {Let $\left\langle S, \mathcal{C}  \right\rangle$ be a CAS}
\item {Its dual is $\Delta (S) = \{U : \, \exists W\in \mathcal{C}\,\&\,\exists H\subseteq W\, \&\, U= W^c\cup H \}$}
\item {If $(\forall A, B\in \Delta(S))\, A\cdot B = A^c\cup B$, then $\left\langle\underline{\Delta(S)}, \cdot, S   \right\rangle$ is a Tarski subalgebra of sets}
\item {The map $\xi_S: S\longmapsto Spec(\Delta(S))$ defined by $\xi_S(x) = \{U:\, x\in U\in \Delta(S)\}$ is injective and a surjection.}
\item {For the converse, if $H$ is a finite Tarski algebra, then define an embedding $\sigma: H \longmapsto \wp (Spec(H))$ via $(\forall x )\, \sigma_H(x) = \{K:\, x\in K\in Spec(H)\}$  into the Tarski algebra of sets on $\wp (Spec(H))$ - the latter is determined by  the coatoms of $H$. Let $\mathcal{K}_H = \{\sigma(x)^c:\, x\in H \}$, then $\left\langle Spec(H), \mathcal{K}_H   \right\rangle$ is the associated set of $H$}
\item {Further $\sigma_H(H) = \Delta(Spec(H))$ and so $H\cong \Delta(Spec(H))$. The computation of $Spec(H)$ is simplified by the fact that every element of it is of the form $(x\downarrow)^c$ for a coatom $x$}
\end{itemize}

If the covering approximation space is infinite, then additional topologies are required to form a duality context with Tarski algebras. The requirements are strong (especially the Hausdorff part). Further only in a few cases does it happen that the upper approximation is topological (see \cite{am501}). 

\subsection{Direct Embedding Theorems}

As noted earlier, in general $\sigma$-algebras are quite different from Boolean algebras. It is useful to know when a covering approximation space or a set with a granulation is embeddable in a $\sigma$ algebra. 

\begin{theorem}
If $\mathcal{S} = \left\langle \underline{\mathcal{S}}, \cup, \cap, ^c , \emptyset    \right\rangle$ is an abstract $\sigma$-algebra, then it is possible to define an implication operation $\cdot$  as below such that $\left\langle\underline{\mathcal{S}}, \cdot   \right\rangle$ is a Tarski algebra:
\[(\forall A, B\in \mathcal{S})\, A\cdot B = A^c\cup B\]
\end{theorem}
\begin{proof}
Obviously the operation $\cdot$ is well defined. The axioms of a Tarski algebra can be directly verified.  
\end{proof}

\begin{definition}
A Tarski algebra obtained as in the above theorem will be referred to as a \emph{abstract $\sigma$- Tarski algebra} ASIA. A Tarski algebra obtained from a concrete $\sigma$-algebra will be said to be a \emph{$\sigma$- Tarski algebra} SIA.
\end{definition}

\begin{theorem}
In both SIA and ASIA, the operation $\cdot$ can be extended in a countably infinite way and if $\{a_i\}$ is an infinite sequence of elements, then \[a_1 (a_2 (a_3 (\dots (a_n(\ldots )\ldots ) = (a_1 a_2 )(a_1 (a_3 (\dots (a_n(\ldots )\ldots ))   \tag{T3+}\]   
\end{theorem}

\section{Squeezed Block Semantics}

The reason for considering this specific rough set approach is because it is more general than classical rough sets and can be viewed from the perspective of all of granular operator spaces, relational rough sets and cover-based rough sets. 

There are a few granular approaches to similarity approximation spaces or tolerance spaces, that is general approximation spaces of the form $S\, =\,\left\langle\underline{S}, T  \right\rangle$ with $T$ being a tolerance relation on $S$. Possible semantics depend on choice of granulation. Some choices of granulations in the context are the following:

\begin{itemize}
\item {The collection $\mathcal{B}$ of blocks (maximal subsets $B$ of $S$ that satisfy $B^2 \subseteq T$),}
\item {The collection of successor $\mathcal{N}$ and predecessor $\mathcal{N}_i$ neighborhoods generated by $T$ and }
\item {The collection $\mathcal{T}\, =\,\{\cap(\Gamma): \, \Gamma \subseteq \mathcal{B}\}$. These will be called the collection of \emph{squeezed blocks}.}
\end{itemize}

In classical rough set theory, the negative region of a set is the lower approximation of the complement of the set. This region is disjoint from the upper approximation of the set in question. An analogous property fails to hold in tolerance spaces (TAS). To deal with this different semantic approaches (to tolerance space ) involving modified upper approximations has been considered in \cite{sw,sw3,am105}. The modified upper approximations are formed from upper approximations by \emph{biting off} a part of it to form \emph{bitten upper approximations}. The new approximations turn out to be disjoint from the negative region of the subset and also possess some nice properties.

The squeezed block approach, as a semantics for a specific tolerance space context, was introduced in \cite{sw3}. The nomenclature is due to the present author. In this approach, taking $\mathcal{T}$ as the set of granules, the authors define the the lower, upper and bitten upper approximation of a $X \subseteq  S$ as follows:
\begin{align*}
X^{l_s} \, =\,\bigcup\{A:\,A \subseteq X\,\&\,A\in \mathcal{T}\}  \tag{sq-lower}\\
X^{u_{s}} \, =\,\bigcup\{A:\,A\,\cap\,X\,\neq\,\emptyset \,\&\, A\in \mathcal{T}\} \tag{sq-upper}\\
X^{u_{sb}} \, =\,\bigcup\{A:\,A\,\cap\,X\,\neq\,\emptyset \,\&\, A\in \mathcal{T}\}\,\setminus (X^{c})^{l} \tag{sqb-upper}
\end{align*}

\begin{theorem}
On the complete Boolean algebra with operators \[\left\langle\underline{\wp(S)}, \cup, \cap, l_s, u_{sb}, ^c, \bot, \top   \right\rangle\] on the powerset $ \wp (S)$ (with $\bot=\emptyset$ and $\top= \S$), all of the following hold:  
\begin{align*}
a^{u_{sb}}= a^{c l_s c}\tag{S5-Dual} \\
a\subseteq b \longrightarrow a^{l_s} \subseteq b^{l_s}\tag{Monotone}\\
\bot^{l_s}\, =\,\bot\, =\,\bot^{u_{sb}} \tag{Bottom}\\
\top^{l_s}= \top^{u_{sb}}\, =\,\top \tag{Top}\\
a^{l_s}\subseteq a \subseteq a^{u_{sb}}   \tag{Reflexive}\\
x^{l_s}\, =\,x^{l_s l_s}  \tag{Idempotence}\\
(a\cap b)^{l_s} \subseteq a^{l_s}\cap b^{l_s}    \tag{L3}\\
a^{l_s}\cup b^{l_s} \subseteq (a\cup b)^{l_s}    \tag{L4}
\end{align*}
\end{theorem}

\begin{theorem}
If the set of definable objects is defined by $\delta (S)\, =\,\{\cup H :\, H\subseteq \mathcal{T}\} $, then all of the following hold:
\begin{align*}
\emptyset, S\in\delta(S)   \tag{Bounds}\\ 
(\forall A, B\in Delta(S)) A\cup B, \, A\cap B \in \delta(S)   \tag{Closure}\\
\left\langle\underline{\delta(S)},\cup, \cap, \emptyset, S   \right\rangle \text{ is a complete ring of subsets }\tag{Ring}
\end{align*}
In fact $\delta(S)$ with the induced operations forms an Alexandrov topology \cite{sw3}. 
\end{theorem}

On the set of definable objects $\delta (S)$, let 
\begin{itemize}
\item{$X\,\rightarrow\,Z \, =\,\bigcup \{B\in \mathcal{T}\,\&\, X \cap B\subseteq Z\}$} 
\item {$X\,\ominus Z \, =\,\bigcap \{B\in \mathcal{T}\,\&\,X\subseteq Z\cup B\}$.}
\end{itemize}
Then the following theorem provides a topological algebraic semantics (\cite{sw3}):
\begin{theorem}\label{gra}
$\left\langle \delta (S),\,\cap,\,\cup,\,\rightarrow,\,\ominus,\,\emptyset,\,S\right\rangle $ is a complete atomic double Heyting algebra. It is also atomistic. 
\end{theorem}
\begin{proof}
\begin{itemize}
\item {By the previous theorem, $\delta(S)$ is also infinitely join and infinitely meet distributive lattice.}
\item {The set $P=\bigcup \{B;\, B\in \mathcal{T}\& X\cap B\subseteq Z\}$ is also the greatest element in the set that satisfies 
$X\cap P \subseteq Z$.}
\item {The set $Q=\bigcap \{B;\, B\in \mathcal{T}\& X\subseteq Z \cup B\}$ is also the least element in the set that satisfies 
$X \subseteq Q\cup Z$.}
\item {The axioms of a complete double Heyting algebra can be verified from this. }
\item {The least granule and definite set containing an element $x$ is $A_x\, =\,\bigcap \{A ; x\in A \in \mathcal{T}\}$. It is also the least neighborhood of $x$ in the Alexandrov topology.}
\item {So $\{A_x: \, x\in S\}$ is the least base for the topology and the set of atoms of the lattice $\delta(S)$. Therefore, every $Z\in \delta(S)$ must contain a set of the form $A_x$. This proves that the double Heyting algebra is atomic.}
\end{itemize}
\end{proof}

\begin{proposition}
$\left\langle \delta (S),\,\cap,\,\cup,\,\rightarrow,\,\ominus,\,\emptyset,\,S\right\rangle $ is not regular and does not satisfy weak law of excluded middle. 
\end{proposition}

\subsection{Squeezed Tolerance Algebra}

The connection with Tarski algebras and spaces is considered next:

\begin{definition}
\begin{itemize}
\item {Let $T_{\mathcal{T}}(S) = \{U:\, (\exists W\in \mathcal{T})(\exists H\subseteq W)\, U= W^c \cup H\}$}
\item {Define a binary operation $\cdot$ by \[(\forall A, B\in T_{\mathcal{T}(S)})\, A\cdot B = A^c \cup B\] }
\item {$\mathfrak{T}(S) = \left\langle\underline{T_{\mathcal{T}}(S)}, \cdot  \right\rangle$ will be called the \emph{presqueezed tolerance algebra}.}
\item {On $\mathfrak{T}(S)$, the operations $l_s,\, u_s, \, u_{bs}$ can be defined.  }
\item {$\mathfrak{T}^*(S) = \left\langle\underline{T_{\mathcal{T}}(S)}, \cdot, l_s, u_{bs},   \right\rangle$ will be called the \emph{squeezed tolerance algebra} (STA).}
\end{itemize}
\end{definition}

\begin{theorem}
In the above definition, 
\begin{itemize}
\item {all of the operations $\cdot, \, l_s,\, u_s,\, u_{bs}$ are well defined.}
\item {A presqueezed measure algebra is a Tarski algebra.}
\item {On $\mathfrak{T}^*(S)$, it is not possible to define a Boolean algebra structure in general.}
\end{itemize}
\end{theorem}

\begin{proof}
\begin{itemize}
\item {That $\mathfrak{T}(S)$ is closed under $\cdot$ follows from the representation of Tarski algebras and definition of its base set.}
\item {From the construction of $T_{\mathcal{T}}(S)$, it is clear that it contains all elements of the form $A^{l_s}, \, A^{u_s}$ and $A^{u_{bs}}$. So the induced operations are well defined.}
\item {The verification of the Tarski algebra axioms is as follows:
\begin{itemize}
\item {$(\forall A\in T_{\mathcal{T}}(S))\,A\cdot A = A^c\cup A = S = 1\in T_{\mathcal{T}}(S)$}
\item {$(\forall A, B, C\in T_{\mathcal{T}}(S))\,(A\cdot B)\cdot (A\cdot C) = (A^c\cup B)^c\cup(B^c\cup C) = A^c\cup (B^c\cup C)$}
\item {$(\forall A, B \in T_{\mathcal{T}}(S))\,(A\cdot B)\cdot B = (A^c\cup B)^c \cup B = (A \cap B^c)\cup B = (A\cup B)\cap (B\cup B^c) = (B\cdot A)\cdot A$.}
\end{itemize}}
\item {The proof of the last assertion is in Thm. \ref{triv}}
\end{itemize}
\end{proof}

\begin{proposition}
$\mathcal{T}$  and $\delta(S)$ are not closed under the operation $\cdot$ in general. 
\end{proposition}

\begin{proof}
The proposition can be proved by easy counterexamples.
\end{proof}

\section{Conclusions: Further Directions}

This study along with the results proved confirm that 
\begin{itemize}
\item {dependencies in rough sets can be compared with those in measurable spaces and specifically in probability.}
\item {relations between dependencies (in a few perspectives) in rough sets can be compared with those in measurable spaces and specifically in probability.}
\item {the comparison requires perspectives of its own - this can be read as success or failure.}
\item {SIA and ASIA can be used for direct embeddings.}
\item {analogical comparison using rough membership functions and subjective or measure theoretic probability functions is not well grounded.}
\item {decision making using rough sets can be improved through the duality based approach of this paper (especially in contexts where ratios of likelihood are more sensible).}
\end{itemize}

To compare dependencies across a possibly enhanced Tarski algebra from a rough semantics with a Tarski algebra from a measurable space, one simply needs to use semi-morphic or morphic embeddings of the former into the latter to specify the \emph{interpretation}. If a whole set of embeddings can be specified, then the comparison would be on a better footing.

The only obvious \emph{relation between dependencies} is that of implication (whatever that means) in the context of the duality. The following meta statement can be asserted.

\begin{proposition}
A dependence of the form $\beta_i(A, B)$ in a Tarski algebra derived from a rough set semantics need not be preserved by a morphism between Tarski algebras in the sense that \[\varphi(\beta_i(A, B) ) = \beta_i (\varphi (A), \varphi (B))\] 
\end{proposition}

A number of open problems in the context of specific rough sets are motivated by the present paper. The number of results on connections between cover based rough sets and admissible topologies is surprisingly low. Improvements and reformulations of topological results for the context are also motivated by the present study. In the logic of rough sets, one encounters various kinds of approximate implications. Related dualities are not mature enough for the complexity of the problem as of this writing.   
Application to various other rough semantics is also motivated by the study. It is hoped that this research will have much impact on interconnections between the diverse fields of rough sets and probability theory.  

\bibliography{algrough}

\end{document}